\documentclass[a4paper,12pt]{article}
\title{\bf{Curve counting theories via stable objects~$\mII$:
DT/ncDT flop formula}
}
\date{}
\author{Yukinobu Toda}

\usepackage{makeidx}

\usepackage{latexsym}
\usepackage{amscd}
\usepackage{amsmath}
\usepackage{amssymb}
\usepackage{amsthm}
\usepackage{float}
\usepackage[all]{xy}
\usepackage{graphicx}

\usepackage{array}
\usepackage{amscd}
\usepackage[all]{xy}
\usepackage{makeidx}
\usepackage{latexsym}
\DeclareFontFamily{U}{rsfs}{%
\skewchar\font127}
\DeclareFontShape{U}{rsfs}{m}{n}{%
<-6>rsfs5<6-8.5>rsfs7<8.5->rsfs10}{}
\DeclareSymbolFont{rsfs}{U}{rsfs}{m}{n}
\DeclareSymbolFontAlphabet
{\mathrsfs}{rsfs}
\DeclareRobustCommand*\rsfs{%
\@fontswitch\relax\mathrsfs}
\theoremstyle{plain}
\newtheorem{thm}{Theorem}[section]
\newtheorem{prop}[thm]{Proposition}
\newtheorem{lem}[thm]{Lemma}

\newtheorem{defi}[thm]{Definition}
\newtheorem{rmk}[thm]{Remark}

\newtheorem{step}{Step}
\newtheorem{sstep}{Step}
\newtheorem{ssstep}{Step}

\newtheorem{prop-defi}[thm]{Proposition-Definition}
\newtheorem{thm-defi}[thm]{Theorem-Definition}
\newtheorem{lem-defi}[thm]{Lemma-Definition}

\newtheorem{assum}[thm]{Assumption}
\newtheorem{conj}[thm]{Conjecture}

\newcommand{\mII}{\mathrm{I}\hspace{-.1em}\mathrm{I}}

\newdimen\argwidth
\def\db[#1\db]{
 \setbox0=\hbox{$#1$}\argwidth=\wd0
 \setbox0=\hbox{$\left[\box0\right]$}
  \advance\argwidth by -\wd0
 \left[\kern.3\argwidth\box0 \kern.3\argwidth\right]}

\newcommand{\kakkoS}{\mathbb{C}\db[S\db]}
\newcommand{\kakkoSl}{\mathbb{C}\db[S_{\lambda}\db]}

\newcommand{\kakkoT}{\mathbb{C}\db[T\db]}

\newcommand{\aA}{\mathcal{A}}
\newcommand{\bB}{\mathcal{B}}
\newcommand{\cC}{\mathcal{C}}
\newcommand{\dD}{\mathcal{D}}
\newcommand{\eE}{\mathcal{E}}
\newcommand{\fF}{\mathcal{F}}

\newcommand{\hH}{\mathcal{H}}

\newcommand{\lL}{\mathcal{L}}
\newcommand{\mM}{\mathcal{M}}

\newcommand{\oO}{\mathcal{O}}
\newcommand{\pP}{\mathcal{P}}
\newcommand{\qQ}{\mathcal{Q}}

\newcommand{\tT}{\mathcal{T}}
\newcommand{\uU}{\mathcal{U}}
\newcommand{\vV}{\mathcal{V}}

\newcommand{\lr}{\longrightarrow}

\newcommand{\Supp}{\mathop{\rm Supp}\nolimits}
\newcommand{\Hom}{\mathop{\rm Hom}\nolimits}

\newcommand{\dR}{\mathbf{R}}

\newcommand{\id}{\textrm{id}}

\newcommand{\ch}{\mathop{\rm ch}\nolimits}
\newcommand{\rk}{\mathop{\rm rk}\nolimits}
\newcommand{\td}{\mathop{\rm td}\nolimits}
\newcommand{\Ext}{\mathop{\rm Ext}\nolimits}
\newcommand{\Spec}{\mathop{\rm Spec}\nolimits}

\newcommand{\Coh}{\mathop{\rm Coh}\nolimits}

\newcommand{\cneq}{\mathrel{\raise.095ex\hbox{:}\mkern-4.2mu=}}
\newcommand{\eqcn}{\mathrel{=\mkern-4.5mu\raise.095ex\hbox{:}}}

\newcommand{\NE}{\mathop{\rm NE}\nolimits}

\newcommand{\gr}{\mathop{\rm gr}\nolimits}

\newcommand{\Aut}{\mathop{\rm Aut}\nolimits}

\newcommand{\Ex}{\mathop{\rm Ex}\nolimits}

\newcommand{\Stab}{\mathop{\rm Stab}\nolimits}

\newcommand{\oPPer}{\mathop{\rm ^{0}Per}\nolimits}

\newcommand{\iPPer}{\mathop{\rm ^{-1}Per}\nolimits}

\newcommand{\ppPPer}{\mathop{^{{p}}\rm{Per}}\nolimits}

\newcommand{\pH}{\mathop{^{{p}}\mathcal{H}}\nolimits}

\newcommand{\pB}{\mathop{^{{p}}\mathcal{B}}\nolimits}
\newcommand{\oB}{\mathop{^{{0}}\mathcal{B}}\nolimits}
\newcommand{\iB}{\mathop{^{{-1}}\mathcal{B}}\nolimits}
\newcommand{\pPhi}{\mathop{^{p}\hspace{-.05em}\Phi}\nolimits}

\newcommand{\ppm}{\mathop{^{p}m}\nolimits}

\newcommand{\pchi}{\mathop{^{p}\chi}\nolimits}

\newcommand{\pF}{\mathop{^{{p}}\mathcal{F}}\nolimits}
\newcommand{\oF}{\mathop{^{{0}}\mathcal{F}}\nolimits}
\newcommand{\iF}{\mathop{^{{-1}}\mathcal{F}}\nolimits}
\newcommand{\pT}{\mathop{^{{p}}\mathcal{T}}\nolimits}
\newcommand{\oT}{\mathop{^{{0}}\mathcal{T}}\nolimits}

\newcommand{\pU}{\mathop{^{{p}}\mathcal{U}}\nolimits}
\newcommand{\oU}{\mathop{^{{0}}\mathcal{U}}\nolimits}
\newcommand{\iU}{\mathop{^{{-1}}\mathcal{U}}\nolimits}

\newcommand{\pmV}{\mathop{^{{p}}\mathcal{V}}\nolimits}
\newcommand{\omV}{\mathop{^{{0}}\mathcal{V}}\nolimits}
\newcommand{\imV}{\mathop{^{{-1}}\mathcal{V}}\nolimits}

\newcommand{\pS}{\mathop{^{{p}}S}\nolimits}

\newcommand{\pTT}{\mathop{^{{p}}T}\nolimits}

\newcommand{\pE}{\mathop{^{p}\hspace{-.05em}\mathcal{E}}\nolimits}
\newcommand{\oE}{\mathop{^{0}\hspace{-.05em}\mathcal{E}}\nolimits}
\newcommand{\iE}{\mathop{^{-1}\hspace{-.05em}\mathcal{E}}\nolimits}
\newcommand{\pV}{\mathop{^{{p}}{V}}\nolimits}
\newcommand{\oV}{\mathop{^{{0}}{V}}\nolimits}
\newcommand{\iV}{\mathop{^{{-1}}{V}}\nolimits}

\newcommand{\pAA}{\mathop{^{p}\hspace{-.2em}A}\nolimits}
\newcommand{\oAA}{\mathop{^{0}\hspace{-.2em}A}\nolimits}

\newcommand{\pPPer}{\mathop{\rm ^{\mathit{p}}Per}\nolimits}

\newcommand{\DT}{\mathop{\rm DT}\nolimits}
\newcommand{\PT}{\mathop{\rm PT}\nolimits}

\newcommand{\End}{\mathop{\rm End}\nolimits}

\newcommand{\Imm}{\mathop{\rm Im}\nolimits}
\newcommand{\imm}{\mathop{\rm im}\nolimits}

\newcommand{\Ker}{\mathop{\rm ker}\nolimits}
\newcommand{\Ree}{\mathop{\rm Re}\nolimits}

\newcommand{\tr}{\mathop{\rm tr}\nolimits}
\newcommand{\ex}{\mathop{\rm ex}\nolimits}

\newcommand{\length}{\mathop{\rm length}\nolimits}
\newcommand{\Sing}{\mathop{\rm Sing}\nolimits}
\newcommand{\cl}{\mathop{\rm cl}\nolimits}
\begin{document}
\maketitle
\begin{abstract}
The goal of the present paper
is to show the
transformation formula of
Donaldson-Thomas invariants on smooth projective
Calabi-Yau 3-folds  
under birational transformations
via categorical method.
We also generalize the non-commutative Donaldson-Thomas 
invariants, introduced by B.~Szendr{\H o}i in a local 
$(-1, -1)$-curve
example, to an arbitrary flopping contraction 
from a smooth projective Calabi-Yau 3-fold.
The transformation formula between such invariants and 
the usual Donaldson-Thomas invariants
are also established. 
These formulas will be deduced 
from the wall-crossing formula 
in the space of weak stability conditions on 
the derived category.  
\end{abstract}
\section{Introduction}
This paper is a sequel of the author's previous 
paper~\cite{Tcurve1}, and study the generating series
of Donaldson-Thomas (DT for short) type invariants 
via categorical method. 
The main result is to show the 
transformation formula of our generating series
under birational transformations of 
Calabi-Yau 3-folds, and the generalized
McKay correspondence introduced 
by Van den Bergh~\cite{MVB}. 
We use the space of weak stability conditions on 
triangulated categories, which generalizes
Bridgeland's stability conditions~\cite{Brs1}, and 
the wall-crossing formula of the generating series
due to Joyce and Song~\cite{JS}, Kontsevich and Soibelman~\cite{K-S}. 

\subsection{Motivation}
Let $X$ be a smooth projective Calabi-Yau 
3-fold over $\mathbb{C}$, i.e. 
\begin{align*}
\bigwedge^3 T_X^{\vee} \cong \oO_X, \quad
H^1(X, \oO_X)=0.
\end{align*}
Let 
$$\phi\colon X^{+} \dashrightarrow X,$$
be a birational map between smooth projective 
Calabi-Yau 3-folds. 
The purpose of this paper is to compare curve counting 
theories on $X$ and $X^{+}$ via categorical method, 
i.e. effectively use an equivalence of bounded 
derived categories of coherent sheaves
by Bridgeland~\cite{Br1},
\begin{align}\label{Deq}
\Phi \colon D^b(\Coh(X^{+})) \stackrel{\sim}{\lr}
 D^b(\Coh(X)).
 \end{align}
The problem of comparing curve counting
invariants under 
birational transformations
 has been studied in~\cite{Morr}, \cite{LiRu}, \cite{LY}, 
\cite{Komi}
for Gromov-Witten invariants and in~\cite{HL}
for DT invariants, via explicit 
calculations or using 
J.~Li's degeneration formula.
The categorical approach for the above problem 
is studied 
by the author in~\cite{ToBPS} for (a kind of approximation of )
Gopakumar-Vafa invariants. In this paper, 
we give a categorical understanding of transformation formula of 
DT invariants 
under birational maps
using the equivalence (\ref{Deq}). 

Recall that a flop is a birational map 
$\phi \colon X^{+} \dashrightarrow X$ which 
fits into a diagram, 
\begin{align}\label{fig:flop0}\xymatrix{
X^{+} \ar[dr]_{f^{+}} \ar@{.>}[rr]^{\phi} & & \ar[dl]^{f} X \\
& Y, &
}\end{align}
where $Y$ is a projective 3-fold with only Gorenstein 
singularities, $f$, $f^{+}$ are birational morphisms 
isomorphic in codimension one, and
the relative Picard numbers of $f$, $f^{+}$ are one
respectively. 
(cf.~Definition~\ref{def:flop}.)
It is well-known that any birational map 
$\phi \colon X^{+} \dashrightarrow X$ between 
smooth projective Calabi-Yau 3-folds is decomposed 
into a composition of flops, thus 
our problem is reduced to the case of a flop. 
 In this case, 
M.~Van den Bergh~\cite{MVB} shows that there is 
a sheaf of non-commutative algebras $A_Y$ on $Y$ 
and a derived equivalence, 
\begin{align}\label{genMc}
\Psi \colon D^b(\Coh(A_Y)) \stackrel{\sim}{\to}
D^b(\Coh(X)). 
\end{align}
(In fact there are two such sheaves of non-commutative 
algebras $\pAA_Y$ for $p=0, -1$. Here we 
put $A_Y=\oAA_Y$. See Theorem~\ref{thm:nc}.) 
We also introduce an analogue of DT-invariant
for the non-commutative scheme $(Y, A_Y)$, 
which generalizes Szendr{\H o}i's
 non-commutative DT
(ncDT for short) invariant for a 
local $(-1, -1)$-curve example.
It 
is introduced in~\cite{Sz} and 
some other local examples are studied in~\cite{Young2}, 
\cite{MR}, \cite{Ng}.
Our invariant is interpreted as a 
globalization of the
local ncDT-invariant.  
We consider the generating series of our invariants,
 and establish the formula 
which relates global ncDT-invariants of $(Y, A_Y)$ 
to usual DT-invariants of $X$ and $X^{+}$. 
This result answers the 
  problem 
addressed by Szendr{\H o}i~\cite[Section~3.5]{Sz}.

\subsection{Donaldson-Thomas theory}
Let us briefly recall the Donaldson-Thomas theory.
For a smooth projective Calabi-Yau 3-fold $X$, take 
$\beta \in H_2(X, \mathbb{Z})$ 
and $n\in \mathbb{Z}$.
Let
$I_n(X, \beta)$ be the Hilbert scheme of 
curves on $X$, 
$$I_{n}(X, \beta)=\left\{\begin{array}{ll} \mbox{ subschemes }
C\subset X, \dim C \le 1 \\
\mbox{ with } [C]=\beta, \ \chi(\oO_C)=n. \end{array}
\right\}.
$$
The moduli space $I_n(X, \beta)$ is projective and has a 
symmetric obstruction theory~\cite{Tho}. The associated 
virtual 
fundamental cycle has virtual dimension zero, and the
integration along it defines
the DT-invariant, 
$$I_{n, \beta}=\int_{[I_n(X, \beta)^{\rm{vir}}]}1 \in \mathbb{Z}.$$
Another way of defining DT-invariant is to use Behrend's 
microlocal function~\cite{Beh}. For an arbitrary scheme $M$, 
Behrend
associates a constructible function, 
$$\nu_{M} \colon M \to \mathbb{Z}. $$
The function $\nu_M$ has the property that if
 $M$ has a symmetric obstruction theory, then 
 the integration of the virtual fundamental cycle coincides 
 with the weighted Euler characteristic, 
\begin{align}\label{weight}
\int_{[M^{\rm{vir}}]}1 =\chi(M, \nu_M) \cneq \sum_{n\in \mathbb{Z}}
n\chi(\nu_M^{-1}(n)).
\end{align}
We consider the generating series, 
\begin{align}\label{DT(X)}
\DT(X)&\cneq\sum_{n, \beta}I_{n, \beta}x^n y^{\beta}, \\
\label{DT_0}
\DT_{0}(X)&\cneq \sum_{n}I_{n, 0}x^n
=M(-x)^{\chi(X)}, \\
\label{DT(X/Y)}
\DT(X/Y)&\cneq \sum_{n, f_{\ast}\beta=0}I_{n, \beta}x^n y^{\beta},
\end{align}
where $f\colon X\to Y$ is a
flopping contraction as in 
the diagram (\ref{fig:flop}), and $M(x)$ is the 
MacMahon function, 
$$M(x)=\prod_{k\ge 1}(1-x^k)^{-k}.$$
The formula (\ref{DT_0}) for
 $\DT_{0}(X)$ is established in~\cite{BBr},
\cite{Li}, \cite{LP}.   
By the MNOP conjecture~\cite{MNOP}, the reduced series 
$$\DT'(X)=\frac{\DT(X)}{\DT_0(X)}, \quad 
\DT'(X/Y)=\frac{\DT(X/Y)}{\DT_0(X)},$$
are expected to coincide with  
the generating series of Gromov-Witten invariants after 
a suitable variable change.

\subsection{Non-commutative Donaldson-Thomas theory}
The following example is worked out by 
Szendr{\H o}i~\cite{Sz}.
Let $Y$ be the conifold singularity, 
$$Y=(xy+zw=0) \subset \mathbb{C}^4,$$
and $f\colon X\to Y$, $f^{+} \colon 
X^{+} \to Y$ 
blow-ups 
at ideals $(x, z)\subset \oO_Y$, 
$(x, w) \subset \oO_Y$ 
respectively. 
This gives an example of 
a (local) flop.
\begin{align}\label{fig:flop2}\xymatrix{
X^{+} \ar[dr]_{f^{+}} \ar@{.>}[rr]^{\phi} & & \ar[dl]^{f} X \\
& Y. &
}\end{align}
The exceptional loci of $f$, $f^{+}$ are smooth 
rational curves $C\subset X$, $C^{+} \subset X^{+}$ 
whose normal bundles are isomorphic to $\oO_{\mathbb{P}^1}(-1)^{\oplus 2}$. 
There is a local version of the
 equivalence (\ref{genMc}), 
 and the $\oO_Y$-algebra $A_Y$ is 
  the path algebra of the following quiver, 
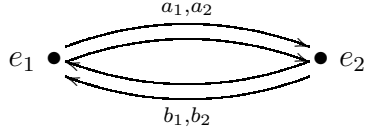
\[ \begin{xy}
(0,0)="\bullet"*{e_1 \ \bullet}, +/r4cm/="B"*{\bullet \ e_2},
\ar^{a_1, a_2}@/^0.3cm/ @<2mm> "\bullet"+/r.4cm/;"B"+/l.4cm/
\ar^{}@/^0.3cm/ "\bullet"+/r.4cm/;"B"+/l.4cm/
\ar^{}@/^0.3cm/  "B"+/l.4cm/;"\bullet"+/r.4cm/
\ar^{b_1, b_2}@/^0.3cm/ @<2mm> "B"+/l.4cm/;"\bullet"+/r.4cm/
\end{xy} \]
with relation defined by the derivatives of the 
potential $W$, where
$$W=a_1 b_1 a_2 b_2 -a_1 b_2 a_2 b_1.$$
For a dimension vector $v=(v_1, v_2)\in \mathbb{Z}^{\oplus 2}$, the
 moduli space of framed $A_Y$-representations 
$M=(M_{\bullet}, a_{\bullet}, b_{\bullet}, u)$
is denoted by $\mM_v$. 
Here 
$v=(\dim M_1, \dim M_2)$ and 
$u\in M_1$ generates $M_1 \oplus M_2$ as an $A_Y$-module. 
Using the potential $W$, it can be shown that 
$\mM_v$ is written as a critical locus of a
holomorphic function on some complex manifold. 
In particular there is a symmetric perfect obstruction theory 
on $\mM_v$, (cf.~\cite[Theorem~1.3.1]{Sz},) hence 
the zero dimensional virtual fundamental cycle on it.  
The integration of the virtual fundamental cycle yields 
the ncDT invariant, 
\begin{align}\label{Ainv}
A_{n, m[C]}=\int_{[\mM_{(n, m+n)}]^{\rm{vir}}}1 \in \mathbb{Z}.
\end{align}
The transformation rule between 
numerical classes on
$X$ and dimension vectors on $A_Y$ is determined by 
the equivalence (\ref{genMc}). 
We have the associated generating series, 
$$\DT_{0}(A_Y)=\sum_{n, m}A_{n, m[C]}x^{n}y^{m}.$$
The following formula is conjectured by Szendr\H{o}i~\cite{Sz}
and proved by Young~\cite{Young1}, Nagao and Nakajima~\cite{NN}.
\begin{thm} 
\emph{\bf{\cite{Young1}, \cite{NN}}}
We have the formula
\begin{align}\notag
\DT_{0}(A_Y)&=M(-x)^2 \prod_{k\ge 1}(1-(-x)^k y)^{k} 
\prod_{k\ge 1}(1-(-x)^k y^{-1})^{k}, \\
\label{DT(A)=product}
&=\DT(X/Y) \cdot \phi_{\ast}\DT'(X^{+}/Y).
\end{align}
Here $\phi_{\ast}$ is the variable change 
$\phi_{\ast}(x, y)=(x, y^{-1})$. 
\end{thm}
Let us return to the situation 
of an arbitrary flopping contraction 
$f\colon X\to Y$ from a smooth projective Calabi-Yau 3-fold $X$. 
For $\beta \in H_2(X, \mathbb{Z})$ and $n\in \mathbb{Z}$, 
we will introduce a global version of the ncDT invariant, 
$$A_{n, \beta}\in \mathbb{Z}, $$
in Definition~\ref{def:glncDT} as a generalization of 
the invariant (\ref{Ainv}). 
The invariant $A_{n, \beta}$ counts cyclic $A_Y$-modules $F$,
satisfying $\dim \Supp \Psi(F) \le 1$ and 
$$[\Psi(F)]=\beta, \quad \chi(\Psi(F))=n, $$
via the equivalence (\ref{genMc}).  
If $f_{\ast}\beta \neq 0$, then such a cyclic 
$A_Y$-module is not of finite dimension as a $\mathbb{C}$-vector 
space. The 
associated generating series are defined by 
\begin{align}
\label{DTAY}
\DT(A_Y)&\cneq\sum_{n, \beta}A_{n, \beta}x^n y^{\beta}, \\
\label{DT0AY}
\DT_{0}(A_Y)&\cneq \sum_{n, f_{\ast}\beta=0}A_{n, \beta}x^n y^{\beta}.
\end{align}

\subsection{Main result}
Let $f\colon X\to Y$ be a flopping contraction 
from a smooth projective Calabi-Yau 3-fold $X$ to 
a singular 3-fold $Y$, and 
$\phi \colon X^{+} \dashrightarrow X$ 
its flop as in the diagram (\ref{fig:flop0}). 
In order to apply Joyce-Song's work~\cite{JS} 
to our situation, we assume the following 
conjecture. 
\begin{conj}\label{conj:BGint}
Let $\mM$ be the moduli stack of 
objects $E\in D^b \Coh(X)$ satisfying 
\begin{align*}
\Ext^{<0}(E, E)=0.
\end{align*} For
any $[E] \in \mM$, 
let $G$ be a maximal reductive 
subgroup in $\Aut(E)$.
Then there exists a $G$-invariant analytic 
open neighborhood $V$ of $0$ in 
$\Ext^1(E, E)$, 
a $G$-invariant holomorphic function $f\colon V\to \mathbb{C}$
with $f(0)=df|_{0}=0$, and a smooth morphism 
of complex analytic stacks
\begin{align*}
\Phi \colon [\{df=0\}/G] \to \mM,
\end{align*}
of relative dimension $\dim \Aut(E)- \dim G$. 
\end{conj}
The above conjecture 
is a derived category version of~\cite[Theorem~5.3]{JS}
and proved for $E \in \Coh(X)$ in~\cite[Theorem~5.3]{JS}.
Also a similar result is announced 
by Behrend-Getzler~\cite{BG}.
Using the arguments of~\cite[Section~5]{Tcurve1}
and assuming Conjecture~\ref{conj:BGint}, 
we will prove the following.

\begin{thm}\label{conj:main}
\emph{\bf{[Theorem~\ref{thm:form}, Theorem~\ref{thm:form2},
Theorem~\ref{thm:global}]}}
Assuming Conjecture~\ref{conj:BGint}, 
we have the following formula, 
\begin{align}
\label{conj2}
\DT(X/Y)&=i\circ \phi_{\ast}\DT(X^{+}/Y), \\
\label{conj1}
\DT_0(A_Y)&=\DT(X/Y) \cdot \phi_{\ast}\DT'(X^{+}/Y), \\
\label{conj3}
\frac{\DT(X)}{\DT(X/Y)}&=
\frac{\DT(A_Y)}{\DT_0(A_Y)}=\phi_{\ast}
\frac{\DT(X^{+})}{\DT(X^{+}/Y)}.
\end{align}
Here the above generating series are 
defined in (\ref{DT(X)}), (\ref{DT(X/Y)}), (\ref{DTAY}), (\ref{DT0AY}), 
and $\phi_{\ast}$, $i$ are variable changes given by 
$\phi_{\ast}(\beta, n)=(\phi_{\ast}\beta, n)$
and $i(\beta, n)=(-\beta, n)$. 
\end{thm}
If we do not assume Conjecture~\ref{conj:BGint}, 
then we at least have the Euler characteristic 
version of Theorem~\ref{conj:main}, 
which will be given in Theorem~\ref{thm:Euler}. 
The Euler characteristic version is
mathematically rigorous at this moment.

Note that (\ref{conj2})
and 
the equality 
$$\frac{\DT(X)}{\DT(X/Y)}=\phi_{\ast}
\frac{\DT(X^{+})}{\DT(X^{+}/Y)},
$$
given in (\ref{conj3})
are proved by J.~Hu and W.~P.~Li~\cite{HL} in the 
case of a flop at a $(-1, -1)$-curve, 
using J.~Li's degeneration formula. 
Also in this case, the equality (\ref{conj1}) 
is just the formula~(\ref{DT(A)=product}). 
The equality~(\ref{conj1}) together with 
the first equality of (\ref{conj3}) yields 
$$\DT(A_Y)=\DT(X)\cdot \phi_{\ast}\DT'(X^{+}/Y), $$
which is interpreted as a global version of 
the formula~(\ref{DT(A)=product}). 

\subsection{Outline of the proof}\label{subsec:out}
Our proof is based on the 
analysis of weak stability conditions on the 
triangulated category, 
$$\dD_X=\langle \oO_X, \Coh_{\le 1}(X) \rangle_{\tr}\subset
D^b(\Coh(X)), $$
i.e. $\dD_X$ is the smallest triangulated subcategory 
of $D^b(\Coh(X))$, which contains $\oO_X$
and $E\in \Coh(X)$ with $\dim \Supp(E)\le 1$. 
The resulting space of weak stability conditions will be
denoted by 
\begin{align*}
\Stab_{\Gamma_{\bullet}}(\dD_X). 
\end{align*}
Let $\phi \colon X^{+} \dashrightarrow X$ be a flop, and 
$\Phi$ a derived equivalence as in (\ref{Deq}). 
Then $\Phi$ restricts to the equivalence 
$\Phi \colon \dD_{X^{+}} \stackrel{\sim}{\to} \dD_X$, hence 
induces the isomorphism, 
\begin{align*}
\Phi_{\ast} \colon 
\Stab_{\Gamma_{\bullet}^{+}}(\dD_{X^{+}}) 
\stackrel{\sim}{\to}
\Stab_{\Gamma_{\bullet}}(\dD_X). 
\end{align*}
The idea is 
as follows: 
we first construct a
certain region, 
\begin{align*}
\uU_{X} \subset \Stab_{\Gamma_{\bullet}}(\dD_X),
\end{align*}
which is interpreted as a neighborhood 
at the large volume limit in terms of 
string theory. For each $\sigma \in \uU_X$, we will 
construct 
the series $\DT(\sigma)$, the 
generating series of 
 DT type invariants counting 
$\sigma$-semistable objects. 
We will show that, 
at a certain limiting point in $\uU_X$, the
series $\DT(\sigma)$ 
coincides with 
the generating series of stable pair invariants introduced 
by Pandharipande-Thomas~\cite{PT}. 
Now we consider the subset, 
\begin{align*}
\Phi_{\ast}\overline{\uU}_{X^{+}} \cup \overline{\uU}_{X}
\subset \Stab_{\Gamma_{\bullet}}(\dD_X). 
\end{align*} 
We will show that the above 
subset is connected, hence we can 
take a path connecting
points in $\Phi_{\ast}\uU_{X^{+}}$
and $\uU_X$, which are 
sufficiently close to each limiting points. 
 Then the wall-crossing formula 
enables us to describe how
$\DT(\sigma)$ varies 
under change of 
$\sigma$ along the path. 
As a result, we can compare 
generating series 
corresponding to limiting points 
in $\Phi_{\ast}\uU_{X^+}$ and $\uU_{X}$,
and obtain a formula relating 
the generating series of stable pair 
invariants on $X^{+}$ and $X$. 
Since generating series of stable 
pair invariants is related to 
that of DT invariants~\cite{Tcurve1}, 
we obtain the desired formula. 
It turns out that 
 the ncDT theory also 
corresponds to a certain point
in $\Stab_{\Gamma_{\bullet}}(\dD_X)$, and 
a similar argument also gives 
a comparison with ncDT invariants.

\subsection{Content of the paper}
In Section~\ref{sec:pre},
we introduce some notions which is
used in this paper. In Section~\ref{sec:weak},
we construct weak stability conditions on 
the triangulated category $\dD_X$. In Section~\ref{sec:semi}, 
we investigate relevant semistable objects. In Section~\ref{sec:proof}, 
we give a proof of Theorem~\ref{conj:main}. 
In Section~\ref{sec:tech}, we prove some technical lemmas.  

\subsection{Acknowledgement}
The author thanks Tom Bridgeland, Dominic Joyce, 
Kentaro Nagao and Richard Thomas for valuable discussions. 
This work is supported by World Premier International 
Research Center Initiative (WPI initiative), MEXT, Japan. 
This work is partially supported by EPSRC
grant EP/F0348461/1. 

\subsection{Notation and convention}
In this paper, all the varieties are defined over 
$\mathbb{C}$. For a triangulated category 
$\dD$, the shift functor is denoted by $[1]$. 
For a set of objects $S\subset \dD$, we denote by 
$\langle S \rangle_{\tr}\subset\dD$ the
smallest triangulated subcategory of $\dD$ which 
contains $S$. Also we denote by $\langle S \rangle_{\ex}
\subset \dD$
the smallest extension closed subcategory of $\dD$
which contains $S$. For an abelian category $\aA$
and a set of objects $S\subset \aA$, the subcategory 
$\langle S \rangle_{\ex}\subset \aA$ is also defined to be 
the smallest extension closed subcategory of $\aA$
which contains $S$. The abelian category of coherent sheaves is denoted 
by $\Coh(X)$. We say $F\in \Coh(X)$ is $d$-dimensional 
if its support is $d$-dimensional. 
The bounded derived category of coherent sheaves 
is denoted by $D^b(\Coh(X))$. For an 
object $E\in D^b(\Coh(X))$ and $i\in \mathbb{Z}$, we 
denote by $\hH^i(E) \in \Coh(X)$ the $i$-th cohomology 
of $E$. 

\section{Preliminaries}\label{sec:pre}
In this section, we introduce some notions which will be 
used in later sections. 
\subsection{Generalities on weak stability 
conditions}
Here we collect definitions and properties 
of weak stability conditions on triangulated 
categories introduced in~\cite[Section~2]{Tcurve1}. 
This is a generalized notion of Bridgeland's 
stability conditions on triangulated categories~\cite{Brs1}. 
Let $\dD$ be a triangulated category, and $K(\dD)$ 
the Grothendieck group of $\dD$. 
We fix a finitely generated free abelian group $\Gamma$, 
and its filtration, 
\begin{align}\label{filt}
0\subsetneq \Gamma_0 \subsetneq \Gamma_1 
\cdots \subsetneq \Gamma_N=\Gamma, 
\end{align}
with each subquotient 
$$\mathbb{H}_i=\Gamma_i/\Gamma_{i-1}, \quad 
(0\le i\le N)$$
a free abelian group. We also fix a group homomorphism, 
 $$\cl \colon K(\dD) \lr \Gamma.$$ 
 We set $\mathbb{H}_i^{\vee}\cneq \Hom_{\mathbb{Z}}(\mathbb{H}_i, \mathbb{C})$
 and fix a norm $\lVert \ast \rVert_i$ 
 on $\mathbb{H}_i \otimes_{\mathbb{Z}}\mathbb{R}$. 
 For an element
 $$Z=\{Z_i\}_{i=0}^{N} \in \prod_{i=0}^{N}\mathbb{H}_i^{\vee}, $$
 and $v\in \Gamma$, we set 
 $$Z(v)\cneq Z_m([v])\in \mathbb{C}, $$
where $0\le m\le N$ satisfies $v\in \Gamma_m \setminus \Gamma_{m-1}$, 
and $[v]$ is the class of $v$ in $\mathbb{H}_m$. 
Here we set $\Gamma_{-1}=\emptyset$. 
Also we define $\lVert v\rVert \cneq \lVert [v] \rVert_{m}$.
Below we write $\cl(E)\in \Gamma$ just as $E\in \Gamma$
when there is no confusion. 
\begin{defi}\emph{
A \textit{weak stability condition} on $\dD$ is a pair 
$\sigma=(Z, \pP)$, 
\begin{align}\label{pair:weak}
Z\in \prod_{i=0}^{N}\mathbb{H}_i^{\vee}, \quad 
\pP(\phi) \subset \dD, \quad (\phi \in \mathbb{R}),
\end{align}
where $\pP(\phi)$ is a full additive subcategory of $\dD$, 
which satisfies the following axiom.}
\begin{itemize}
\item \emph{For any $\phi \in \mathbb{R}$, we have 
$\pP(\phi)[1]=\pP(\phi+1)$.}
\emph{\item For $E_i \in \pP(\phi_i)$ with $\phi_1>\phi_2$, we have 
$\Hom(E_1, E_2)=0$. }
\emph{\item {\bf(Harder-Narasimhan property):}
For any object $E\in \dD$, we have the 
following collection of triangles:}
$$\xymatrix{
0=E_0 \ar[rr]  & &E_1 \ar[dl] \ar[rr] & & E_2 \ar[r]\ar[dl] & \cdots \ar[rr] & & E_n =E \ar[dl]\\
&  F_1 \ar[ul]^{[1]} & & F_2 \ar[ul]^{[1]}& & & F_n \ar[ul]^{[1]}&
}$$
\emph{such that $F_j \in \pP (\phi _j)$ with $\phi _1 > \phi _2 > \cdots >\phi _n$. }
\item \emph{For any non-zero $E\in \pP(\phi)$, we have} 
\begin{align}\label{phase}
Z(E) \in \mathbb{R}_{>0}\exp(i\pi \phi).
\end{align}
\item \emph{{\bf(Support property):}
There is a constant $C>0$ such that for any 
non-zero $E\in \bigcup_{\phi\in \mathbb{R}}\pP(\phi)$, we have}
\begin{align}\label{support}
\lVert E \rVert \le C\lvert Z(E) \rvert.
\end{align}
\item \emph{{\bf(Local finiteness condition):}
There exists $\eta>0$ such that for any $\phi \in \mathbb{R}$,
the quasi-abelian category 
$\pP((\phi-\eta, \phi+\eta))$ is of finite length.} 
\end{itemize} 
\end{defi}
Here for an interval $I\subset \mathbb{R}$, the subcategory 
$\pP(I)\subset \dD$ and the subset $C_{\sigma}(I)\subset \Gamma$
are defined to be 
\begin{align}\notag
\pP(I)&=\langle \pP(\phi) : \phi \in I \rangle_{\ex} \subset \dD, \\
\label{Csigma}
C_{\sigma}(I)&=\imm\{ \cl \colon \pP(I) \to \Gamma\}.
\end{align}
If $I=(a, b)$ with $b-a<1$, then 
$\pP(I)$ is a quasi-abelian category
(cf.~\cite[Definition~4.1]{Brs1},)
and $\pP(I)$ is said to be of finite length 
if $\pP(I)$ is noetherian and artinian with respect to 
strict epimorphisms and strict morphisms. 
See~\cite[Section~4]{Brs1} for more detail. 

Another way of defining weak stability conditions
is using t-structures. 
The readers can refer~\cite{Brs1} for 
bounded t-structures, and their hearts.  
\begin{defi}\label{def:wefu}
\emph{
Let $\aA\subset \dD$ be the heart of a bounded t-structure on 
a triangulated category 
$\dD$. We say $Z \in \prod_{i=0}^{N}\mathbb{H}_i^{\vee}$
is a \textit{weak stability function} on $\aA$ if for any non-zero $E\in \aA$, 
we have} 
\begin{align}\label{def:weak}
Z(E) \in \mathfrak{H}\cneq
\{ r\exp(i\pi \phi) : r>0, \ 0<\phi \le 1\}.
\end{align}
\end{defi}
By (\ref{def:weak}), we can uniquely 
determine the argument,
$$\arg Z(E) \in (0, \pi],$$
for any $0\neq E\in \aA$. 
For an exact sequence $0\to F\to E \to G\to 0$
in $\aA$, 
one of the
following equalities holds.
\begin{align*}
&\arg Z(F) \le \arg Z(E) \le \arg Z(G), \\
&\arg Z(F) \ge \arg Z(E) \ge \arg Z(G).
\end{align*}
\begin{defi}\emph{
Let $Z\in \prod_{i=0}^{N}\mathbb{H}_i^{\vee}$
be a weak stability function on $\aA$. We say 
$0\neq E \in \aA$ is $Z$-\textit{semistable} (resp. \textit{stable}) if 
for any exact sequence $0 \to F \to E \to G\to 0$ we have} 
\begin{align}\label{argg}
\arg Z(F)\le \arg Z(G), \quad (\mbox{\rm{resp}.~}\arg Z(F)<\arg Z(G).)
\end{align}
\end{defi}
The notion of Harder-Narasimhan filtration is defined 
in a similar way to usual stability conditions. 
\begin{defi}\emph{
Let $Z\in \prod_{i=0}^{N}\mathbb{H}_i^{\vee}$
be a weak stability function on $\aA$.
A \textit{Harder-Narasimhan filtration} of an object $E\in \aA$ 
is a filtration 
$$0=E_0 \subset E_1 \subset \cdots \subset E_{k-1}\subset E_k=E, $$
such that each subquotient $F_j=E_j/E_{j-1}$ is 
$Z$-semistable with 
$$\arg Z(F_1)>\arg Z(F_2)>\cdots >\arg Z(F_k).$$
A weak stability function $Z$
is said to have the \textit{Harder-Narasimhan property} if any 
object $E\in \aA$ has a Harder-Narasimhan filtration.} 
\end{defi}
We will use the following proposition. 
(cf.~\cite[Proposition~2.12]{Tcurve1}.)
\begin{prop}\label{suHN}
Let $Z\in \prod_{i=0}^{N}\mathbb{H}_i^{\vee}$
be a weak stability function on $\aA$.
Suppose that the following chain conditions are satisfied. 

(a) There are no infinite sequences of subobjects in $\aA$, 
$$\cdots \subset E_{j+1} \subset E_j \subset \cdots \subset E_2 \subset E_1$$
with $\arg Z(E_{j+1})>\arg Z(E_j/E_{j+1})$ for all $j$. 

(b) There are no infinite sequences of quotients in $\aA$, 
$$E_1 \twoheadrightarrow E_2 \twoheadrightarrow \cdots \twoheadrightarrow 
E_j \stackrel{\pi_j}{\twoheadrightarrow}
 E_{j+1} \twoheadrightarrow \cdots$$
with $\arg Z(\Ker \pi_j)>\arg Z(E_{j+1})$ for all $j$. 

Then $Z$ has the Harder-Narasimhan property. 
\end{prop}
We have the following proposition. 
(cf.~\cite[Proposition~5.3]{Brs1}, \cite[Proposition~2.13]{Tcurve1}.)
\begin{prop}\label{prop:corr}
Giving a pair 
$(Z, \pP)$
as in (\ref{pair:weak}) 
satisfying (\ref{phase})
 is equivalent to 
giving a bounded t-structure $\aA$ on 
$\dD$ and a weak 
stability function on its heart with the Harder-Narasimhan 
property. 
\end{prop}
The correspondence in the above proposition is 
given as follows. Given a pair $(Z, \pP)$
satisfying (\ref{phase}), we set 
$\aA$ to be 
$$\aA=\pP((0, 1]) \subset \dD.$$
Conversely given the heart of a bounded
t-structure $\aA\subset \dD$ and a weak 
stability function $Z$ on $\aA$, we set 
$\pP(\phi)$ to be
the following full additive subcategory of $\aA$, 
$$\pP(\phi)=\left\{ E\in \aA : \begin{array}{l}
 E\mbox{ is }Z\mbox{-semistable with }\\
 Z(E) \in \mathbb{R}_{>0}\exp(i\pi \phi)
 \end{array}\right\}.$$
 Below we write an element of $\Stab_{\Gamma_{\bullet}}(\dD)$
 either as $(Z, \pP)$ or $(Z, \aA)$, where 
 $(Z, \pP)$ is a pair (\ref{pair:weak}) and $(Z, \aA)$ is 
 given as in Definition~\ref{def:wefu}. 
 Let $\Stab_{\Gamma_{\bullet}}(\dD)$ be the 
 set of weak stability conditions on $\dD$. 
 The following theorem is given in~\cite[Theorem~2.15]{Tcurve1}, 
 as an analogue of~\cite[Theorem~7.1]{Brs1}. 
 \begin{thm}\label{thm:stab}
 \emph{\bf{\cite[Theorem~2.15]{Tcurve1}}}
 The map
$$\Pi\colon \Stab_{\Gamma_{\bullet}}(\dD)
\ni(Z, \pP) \longmapsto Z 
\in \prod_{i=0}^{N}\mathbb{H}_i^{\vee}, $$
is a local homeomorphism. In particular 
each connected component of $\Stab_{\Gamma_{\bullet}}(\dD)$ 
is a complex manifold. 
\end{thm}
We will need the following two lemmas. 
The first one is proved in~\cite[Lemma~7.1]{Tcurve1}. 
\begin{lem}\label{lem:scont}
\emph{\bf{\cite[Lemma~7.1]{Tcurve1}}}
Let $\aA$ be the heart of a bounded t-structure on 
$\dD$, and $(\tT, \fF)$ a torsion pair 
(cf.~Definition~\ref{def:torsion},)
on $\aA$. 
Let $\bB=\langle \fF[1], \tT \rangle_{\ex}$ the 
associated tilting. (cf.~(\ref{dag}).)
Let
$$[0, 1) \ni t \longmapsto Z_t \in 
\prod_{i=0}^{N}\mathbb{H}_i^{\vee}, $$
be a continuous map such that
$\sigma_t=(Z_t, \aA)$ for $0<t<1$ and 
$\sigma_0=(Z_0, \bB)$ determine points in 
$\Stab_{\Gamma_{\bullet}}(\dD)$.
Then we have $\lim_{t\to 0}\sigma_t=\sigma_0$. 
\end{lem}
The second one is a compatibility of 
the weak stability conditions via 
equivalences of triangulated categories. 
The proof is straightforward, 
and we omit the proof. 
\begin{lem}\label{lem:group}
 Let $\dD'$ be another triangulated category together with 
similar additional data $\cl' \colon K(\dD') \to \Gamma'$ 
and a filtration $\Gamma_{\bullet}'$ as in (\ref{filt}). 
Suppose that $\Phi\colon \dD \to \dD'$ gives an equivalence 
of triangulated categories such that there is a
filtration preserving isomorphism $\Phi_{\Gamma}
\colon \Gamma_{\bullet} \to \Gamma'_{\bullet}$ 
which fits into the following commutative diagram,
\begin{align}\label{commutes}
\xymatrix{
K(\dD) \ar[r]^{\Phi}\ar[d]_{\cl} & K(\dD')\ar[d]^{\cl'} \\
\Gamma \ar[r]^{\Phi_{\Gamma}} & \Gamma'.
}
\end{align}
Then there is an isomorphism $\Phi_{\ast}\colon 
\Stab_{\Gamma_{\bullet}}(\dD) \to \Stab_{\Gamma_{\bullet}'}(\dD')$
such that the following diagram commutes, 
$$\xymatrix{\Stab_{\Gamma_{\bullet}}(\dD) \ar[r]^{\Phi_{\ast}}\ar[d]_{\Pi}
& \Stab_{\Gamma_{\bullet}'}(\dD')\ar[d]^{\Pi'} \\
\prod_{i=0}^{N} \mathbb{H}_i^{\vee} 
\ar[r]^{(\gr\Phi_{\Gamma}^{-1})^{\vee}} 
& \prod_{i=0}^{N}\mathbb{H}_i^{'\vee}.}$$
\end{lem}

\subsection{Terminology from birational geometry}
In what follows, we assume that $X$ is a smooth 
projective Calabi-Yau 3-fold over $\mathbb{C}$, 
i.e. 
\begin{align*}
\bigwedge^3 T_{X}^{\vee} \cong \oO_X, \quad 
H^1(X, \oO_X)=0. 
\end{align*}
Here we introduce standard terminology in 
birational geometry, for example used in~\cite[Definition 1.1]{Ka2}. 

Let 
$S$ be a projective variety with a morphism 
$f\colon X\to S$.
Two divisors $D_1$, $D_2$ on $X$ are called \textit{numerically equivalent} 
over $S$ if and only if 
$D_1 \cdot C=D_2 \cdot C$ for any curve $C\subset X$
with $f_{\ast}[C]=0$. 
Similarly,  
one-cycles $C_1$, $C_2$ on $X$ contracted by $f$
 are \textit{numerically equivalent} 
if and only if $D\cdot C_1=D\cdot C_2$ for every divisor $D$ on $X$. 
\begin{defi}\emph{
We define abelian groups $N^1 (X/S)$, $N_1(X/S)$ to be
\begin{align*}
N^1 (X/S) & \cneq \{ \emph{\rm{Divisors on} }X \} / (
\emph{\rm{numerical equivalence over} }S ), \\
 N_1 (X/S) & \cneq \{ \emph{\rm{One-cycles on} }X
  \emph{ \rm{contracted by} }f \} / (
\emph{\rm{numerical equivalence}}).
\end{align*}}
\end{defi}
By the definition, there is the perfect pairing, 
$$N^1 (X/S)_{\mathbb{R}} \times N_1 (X/S)_{\mathbb{R}}
 \ni (D, C) \longmapsto D\cdot C \in \mathbb{\mathbb{R}}.$$
 \begin{defi}\label{amp}\emph{
We define
the \textit{ample cone} $A(X/S)$, the
\textit{complexified ample cone} $A(X/S)_{\mathbb{C}}$,
and the \textit{semigroup of effective 
one-cycles} $\NE(X/S)$ 
 to be}
\begin{align*}
 A(X/S) & \cneq \{ \emph{Numerical classes of }f
\emph{-ample }\mathbb{R} 
\emph{-divisors } \} \subset N^1 (X/S)_{\mathbb{R}}, \\
A(X/S)_{\mathbb{C}} & \cneq \{ B +i \omega \in N^1 (X/S)_{\mathbb{C}} 
: \omega \in A(X/S) \}, \\
\NE(X/S)&\cneq \{ \emph{Effective one-cycles 
contracted by }f\} \subset N_1(X/S). 
 \end{align*}
\end{defi}
For $\beta, \beta'\in N_1(X/S)$, we write 
$\beta \ge \beta'$ if $\beta-\beta' \in \NE(X/S)$. 
   When $S=\Spec\mathbb{C}$, we write
 $$N^1(X)\cneq N^1(X/\Spec \mathbb{C}), \quad 
N_1(X)\cneq N_1(X/\Spec \mathbb{C}), $$
etc, for simplicity. 
We set $N_{\le 1}(X)$ to be 
$$N_{\le 1}(X)\cneq \mathbb{Z}\oplus N_1(X).$$
\begin{defi}\label{def:flopping}\emph{
A birational morphism $f\colon X\to Y$ is called 
a \textit{flopping contraction} if the following 
conditions are satisfied. }
\begin{itemize}
\item \emph{$f$ is isomorphic in codimension one, and 
$Y$ has only Gorenstein singularities.}
\item \emph{We have $\dim_{\mathbb{R}} N^1(X/Y)_{\mathbb{R}}=1$.}
\end{itemize}
\end{defi}
Let $f\colon X\to Y$ be a flopping contraction.  
The exceptional locus $C\subset X$ is a tree of rational 
curves, 
$$C=C_1\cup C_2 \cup \cdots \cup C_N, \quad 
C_i \cong \mathbb{P}^1.$$
(See for example~\cite[Lemma 3.4.1]{MVB}.) 
By the second condition of Definition~\ref{def:flopping}, 
there is a relative ample divisor $H$ on $X$ 
such that 
\begin{align}\label{N^1}
N^1(X/Y)=\mathbb{R}[H], \quad A(X/Y)=\mathbb{R}_{>0}[H].
\end{align}
\begin{defi}\label{def:flop}\emph{
Let $f\colon X\to Y$ be a flopping contraction. 
A \textit{flop} of $f$ is 
a birational map $\phi \colon X^{+} \dashrightarrow X$, 
which fits into a diagram
\begin{align}\label{fig:flop}\xymatrix{
X^{+} \ar[dr]_{f^{+}} \ar@{.>}[rr]^{\phi} & & \ar[dl]^{f} X \\
& Y, &
}\end{align}
such that $f^{+}$ is also a flopping contraction with
$f\circ \phi=f^{+}$, and $\phi$ is not an isomorphism. 
}
\end{defi}
It is well-known that a flop is unique if it exists, and 
any birational map between smooth projective Calabi-Yau 
3-folds is decomposed into a finite number of flops.
(cf.~\cite[Theorem~1]{Kawaflo}.) 
For a flop $\phi \colon X^{+} \dashrightarrow X$, 
we have the linear isomorphisms, 
\begin{align}
\label{strict1}
\phi_{\ast}&\colon N^1(X^{+}/Y)_{\mathbb{R}} \stackrel{\cong}{\lr}
 N^1(X/Y)_{\mathbb{R}}, \\
\label{strict2}
\phi_{\ast}&\colon N_1(X^{+}/Y)_{\mathbb{R}} \stackrel{\cong}{\lr}
 N_1(X/Y)_{\mathbb{R}}, 
\end{align}
where (\ref{strict1}) is the strict transform of divisors, and 
(\ref{strict2}) is the inverse of the dual of (\ref{strict1}). 
Note that $\phi_{\ast}$ takes $A(X^{+}/Y)$ to $-A(X/Y)$ and 
takes $\NE(X^{+}/Y)$ to $-\NE(X/Y)$.

\subsection{t-structures and tilting}
Let $\dD$ be a triangulated category, 
and $\aA \subset \dD$ the heart of a bounded 
t-structure on $\dD$. 
Here we recall the notion of torsion pairs and tilting. 
\begin{defi}\label{def:torsion}
\emph{\bf{\cite{HRS}}}
\emph{Let 
$(\tT, \fF)$ be a pair of full subcategories of $\aA$. 
We say $(\tT, \fF)$ is a \textit{torsion pair}
if the following conditions hold.} 
\begin{itemize}
\item \emph{$\Hom(T, F)=0$ for any $T\in \tT$ and 
$F\in \fF$.}
\item \emph{Any object $E\in \aA$
fits into an exact sequence, 
\begin{align}\label{fits}
0 \lr T \lr E \lr F \lr 0, 
\end{align}
with $T\in \tT$ and $F\in \fF$.}
\end{itemize}
\end{defi}
Given a torsion pair $(\tT, \fF)$ on 
$\aA$, its \textit{tilting} is 
defined by 
\begin{align}\label{dag}
\aA^{\dag}\cneq \left\{ E\in \dD : \begin{array}{l}
\hH^{-1}_{\aA}(E) \in \fF, \ \hH^{0}_{\aA}(E) \in \tT, \\
\hH^i_{\aA}(E)=0 \mbox{ for }i\notin \{-1, 0\}.
\end{array}
\right\},
\end{align}
i.e. $\aA^{\dag}=\langle \fF[1], \tT \rangle_{\ex}$ in 
$\dD$. 
Here $\hH^i_{\aA}(\ast)$ is the $i$-th cohomology 
functor with respect to the t-structure with heart $\aA$. 
It is known that $\aA^{\dag}$ is the 
heart of a bounded t-structure on $\dD$. 
(cf.~\cite[Proposition~2.1]{HRS}.)
Later we will need the following lemma. 
\begin{lem}\label{lem:later}
Let $\aA\subset \dD$ be the heart of a bounded t-structure 
on a triangulated category $\dD$. 
Suppose that $\aA$ is a noetherian abelian category. 

(i) Let $\tT \subset \aA$ be a full subcategory which 
is closed under extensions and quotients in $\aA$. 
Then for $\fF=\{E\in \aA : \Hom(\tT, E)=0\}$, the pair 
$(\tT, \fF)$ is a torsion pair on $\aA$. 

(ii) For a torsion pair $(\tT, \fF)$ on $\aA$, 
suppose that there is no infinite sequence in $\aA$,
\begin{align}\notag
E_0 \hookleftarrow E_1 \hookleftarrow \cdots E_j
\hookleftarrow E_{j+1} \hookleftarrow \cdots, 
\end{align}
with $E_j/E_{j+1}\notin \tT$. 
Let $\fF' \subset \fF$ be a full subcategory which 
is closed under extensions and subobjects in $\aA$. 
Then for $\tT'=\{E\in \aA : \Hom(E, \fF')=0\}$, the pair 
$(\tT', \fF')$ is a torsion pair on $\aA$. 
\end{lem}
\begin{proof}
(i)  
Take $E\in \aA$ with $E\notin \fF$. Then there is 
$T\in \tT$ and a non-zero morphism 
$T\to E$. Since $\tT$ is closed under quotients, 
we may assume that $T\to E$ is a monomorphism in $\aA$. 
Take an exact sequence in $\aA$, 
\begin{align}\label{easylem}
0 \lr T \lr E \lr F \lr 0.
\end{align}
By the noetherian property of $\aA$ and the assumption
that $\tT$ is closed under extensions, we may assume that 
there is no $T\subsetneq T' \subset E$ with $T' \in \tT$. 
Then we have $F\in \fF$ and (\ref{easylem}) gives the 
desired sequence (\ref{fits}). 

(ii) For $E \in \aA$, we have an exact sequence
\begin{align*}
0 \to T \to E \to F \to 0,
\end{align*}
with $T \in \tT$ and $F \in \fF$. 
Using the assumption, the dual argument of (i)
shows that there is an exact sequence in $\aA$, 
\begin{align*}
0 \to F_1 \to F \to F_2 \to 0,
\end{align*}
with $F_1 \in \fF \cap \tT'$ and $F_2 \in \fF'$. 
Combining the above two exact sequences, we obtain 
an exact sequence, 
\begin{align*}
0 \to T' \to E \to F' \to 0,
\end{align*}
with $T' \in \tT'$ and $F' \in \fF'$. 
Therefore $(\tT', \fF')$ is a torsion pair on 
$\aA$. 
\end{proof}

\subsection{Notation of abelian categories}
Here we give some notation of abelian categories 
which will be used in this paper. 
\begin{defi}\emph{
Let $A$ be a sheaf of $\oO_X$-algebras 
on a variety $X$, which is coherent as an $\oO_X$-module. 
 We denote by $\Coh(A)$ the 
abelian category of right coherent $A$-modules. 
For an object $E\in \Coh(A)$, the support of $E$ is defined to be
the support of $E$ as an $\oO_X$-module. We set 
\begin{align*}
\Coh_{0}(A)&=\{ E\in \Coh(A) : \dim \Supp(E)=0\}, \\
\Coh_{\le 1}(A)&=\{E\in \Coh(A) : \dim \Supp(E)\le 1\}, \\
\Coh_{\ge 2}(A)&=\{E\in \Coh(A): \Hom(\Coh_{\le 1}(A), E)=0\}.
\end{align*}
If $A=\oO_X$, we write $\Coh_{\bullet}(\oO_X)$ as $\Coh_{\bullet}(X)$.}
\end{defi}
By Lemma~\ref{lem:later}, 
the pair $(\Coh_{\le 1}(A), \Coh_{\ge 2}(A))$
is a torsion pair on $\Coh(A)$.
\begin{defi}\label{dagga}\emph{
We define $\Coh^{\dag}(A)$ to be the tilting with 
respect to $(\Coh_{\le 1}(A), \Coh_{\ge 2}(A))$, i.e. 
$$\Coh^{\dag}(A)=\langle \Coh_{\ge 2}(A)[1], \Coh_{\le 1}(A)\rangle_{\ex}.$$
If $A=\oO_X$, we write $\Coh^{\dag}(\oO_X)$ as $\Coh^{\dag}(X)$.}\end{defi}
\subsection{Derived equivalence under flops}\label{subsec:der}
Let $f\colon X\to Y$ be a flopping contraction 
from a smooth projective Calabi-Yau 3-fold $X$. 
(cf.~Definition~\ref{def:flopping}.)
In this situation, Bridgeland~\cite{Br1}
associates the subcategories $\ppPPer(X/Y)$ on 
$D^b(\Coh(X))$
for $p=0, -1$, as follows. 
\begin{defi}\label{def:pp}
\emph{
We define $\ppPPer(X/Y) \subset D^b(\Coh(X))$
for $p=0, -1$ to be 
$$\ppPPer(X/Y)=\left\{ E\in D^b(\Coh(X)) : \begin{array}{l}
\dR f_{\ast}E \in \Coh(Y), \\
\Hom^{<-p}(E, \cC)=\Hom^{<-p}(\cC, E)=0.\end{array}\right\},$$
where $\cC=\{ F\in \Coh(X) \mid \dR f_{\ast}F=0\}$.
We also define $\ppPPer_{0}(X/Y)$ and $\ppPPer_{\le 1}(X/Y)$ to be 
\begin{align*}
\ppPPer_{0}(X/Y)&=\{ E\in \ppPPer(X/Y) : \dim \Supp \dR f_{\ast}E=0\}, \\
\ppPPer_{\le 1}(X/Y)&=\{E\in \ppPPer(X/Y): \dim \Supp(E) \le 1\}.
\end{align*}} 
\end{defi}
\begin{rmk}
By the definition, it is easy to see that 
$$\oO_X \in \ppPPer(X/Y), \quad p=0, -1.$$
\end{rmk}
It is proved in~\cite{Br1} that $\ppPPer(X/Y)$ are the 
hearts of bounded t-structures on $D^b(\Coh(X))$, hence 
they are abelian categories. 
The categories $\ppPPer_{\le 1}(X/Y)$ are also the hearts
of bounded t-structures on $D^b(\Coh_{\le 1}(X))$. 
(cf.~\cite[Proposition~5.2]{ToBPS}.)
The generators of $\ppPPer_{\le 1}(X/Y)$ are described as 
follows. 
Let $C_1, \cdots, C_N \subset X$ be the 
irreducible components of the exceptional locus of $f$.
We have the following. 
\begin{lem}\label{lem:straight}
The abelian
categories $\ppPPer_{\le 1}(X/Y)$ are described as 
\begin{align}\label{per:gen1}
\oPPer_{\le 1}(X/Y)&=
\langle \omega_{f^{-1}(y)}[1], \oO_{C_i}(-1), \widetilde{\Coh}_{\le 1}(X)
\rangle_{\ex}, \\ 
\label{per:gen2}
\iPPer_{\le 1}(X/Y)&=
\langle \oO_{f^{-1}(y)}, \oO_{C_i}(-1)[1], \widetilde{\Coh}_{\le 1}(X)
\rangle_{\ex}. 
\end{align}
Here $y\in \Sing(Y)$, $1\le i\le N$, and $\widetilde{\Coh}_{\le 1}(X)$
is defined to be 
$$\widetilde{\Coh}_{\le 1}(X)\cneq \{F\in \Coh_{\le 1}(X) \mid 
C_i \nsubseteq \Supp(F) \mbox{ for all }i\}.$$
\end{lem}
\begin{proof}
This is a straightforward generalization of~\cite{MVB}
and the proof is written in~\cite[Proposition~5.2]{ToBPS}.
\end{proof}
Let
 $\phi \colon X^{+} \dashrightarrow X$ be the flop of 
$f$. (cf.~Definition~\ref{def:flop}.)
The following theorem is proved in~\cite{Br1}. 
\begin{thm}\emph{\bf{\cite{Br1}}}
There is an equivalence of bounded
 derived categories of coherent sheaves, 
\begin{align}\label{standard}
\Phi\colon D^b(\Coh(X^{+})) \stackrel{\sim}{\lr} D^b(\Coh(X)), 
\end{align}
which takes $\iPPer(X^{+}/Y)$ to $\oPPer(X/Y)$. 
\end{thm}

\subsection{Flops and non-commutative algebras}
Let $f\colon X\to Y$ be a flopping contraction
as in Definition~\ref{def:flopping}. 
By Van den Bergh~\cite{MVB}, 
the abelian categories $\ppPPer(X/Y)$ are related to 
sheaves of non-commutative algebras on $Y$. 
\begin{thm}\emph{\bf{\cite{MVB}}}\label{thm:nc}
There 
are vector bundles $\pE$ 
on $X$ for $p=0, -1$, which admit derived equivalences, 
\begin{align}\label{nc}
\pPhi=\dR f_{\ast} \dR \hH om(\pE, \ast) \colon 
D^b(\Coh(X)) \stackrel{\cong}{\lr} D^b(\Coh(\pAA_{Y})).
\end{align}
Here $\pAA_{Y}=f_{\ast}\eE nd(\pE)$
 are sheaves of non-commutative algebras on $Y$.
  The equivalences (\ref{nc}) restrict to equivalences 
 \begin{align}\label{restper}
 \pPhi \colon \ppPPer(X/Y) \stackrel{\sim}{\lr} \Coh(\pAA_{Y}).
 \end{align}
 \end{thm}
 \begin{proof}
 Here we briefly recall how to construct $\pE$,
 which will be needed in the later section. 
 We treat the case of $p=0$ for simplicity. 
 Let $\lL_{X}$ be 
 a globally generated ample line bundle on $X$. 
 We have a surjection of sheaves,
 $$(\lL_{Y}^{-1})^{\oplus a}\twoheadrightarrow R^1 f_{\ast}\lL_X^{-1}, $$
for a sufficiently ample line bundle $\lL_Y$ on $Y$ 
and $a>0$. Taking the adjunction, we obtain the 
short exact sequence, 
\begin{align}\label{oE}
0 \lr \lL_X^{-1}  \lr \oE' \lr f^{\ast}(\lL_Y^{-1})^{\oplus a}\lr 0.
\end{align}
Then $\oE$ is defined to be 
$$\oE=\oO_X \oplus \oE'.$$
 The constructions of $\iE'$ and 
 $\iE$ are similar. 
(See~\cite{MVB} for the detail.)
\end{proof}
\begin{rmk}
By the construction, the sheaves of algebras $\pAA_{Y}$
are direct sums of locally projective 
$\pAA_{Y}$-modules, 
\begin{align}\label{locpro}
\pAA_{Y}=\pAA_{Y}'\oplus \pAA_{Y}'',
\end{align}
where $\pAA_{Y}'=\pPhi(\oO_X)$ and 
$\pAA_{Y}''=\pPhi(\pE')$. 
\end{rmk}
Note that the torsion pair $(\Coh_{\le 1}(\pAA_{Y}), \Coh_{\ge 2}(\pAA_{Y}))$
induces the torsion pair 
$$(\ppPPer_{\le 1}(X/Y), \ppPPer_{\ge 2}(X/Y)),$$
on $\ppPPer(X/Y)$ via the equivalence 
$\pPhi\colon \ppPPer(X/Y) \stackrel{\sim}{\to} \Coh(\pAA_Y)$. 
\begin{defi}\label{def:perdag}
\emph{
We define the abelian category $\pPPer^{\dag}(X/Y)$ to be
the tilting with respect to the torsion pair 
$(\ppPPer_{\le 1}(X/Y), \ppPPer_{\ge 2}(X/Y))$, i.e. 
$$\pPPer^{\dag}(X/Y)=\langle \ppPPer_{\ge 2}(X/Y)[1], \ppPPer_{\le 1}(X/Y)
\rangle_{\ex}.$$
}\end{defi}
\begin{rmk}
By the construction, the equivalence $\pPhi \colon D^b(\Coh(X))
\stackrel{\sim}{\to} D^b(\Coh(\pAA_Y))$ 
restricts to the equivalence between 
$\ppPPer^{\dag}(X/Y)$ and $\Coh^{\dag}(\pAA_Y)$. 
\end{rmk}
\subsection{Donaldson-Thomas theory}
Here we introduce Donaldson-Thomas invariants. 
For $(n, \beta)\in \mathbb{Z}\oplus N_1(X)$, 
let $I_n(X, \beta)$ be the moduli space of 
subschemes $C\subset X$ with 
$$\dim C \le 1, \quad [C]=\beta, \quad \mbox{and} \quad 
\chi(\oO_C)=n.$$
There is a
symmetric perfect obstruction theory on $I_n(X, \beta)$~\cite{Tho}, 
and the associated virtual cycle,
$$[I_n(X, \beta)]^{\rm{vir}} \in A_{0}(I_n(X, \beta)).$$
\begin{defi}\emph{
The \textit{Donaldson-Thomas invariant} 
is defined by 
$$I_{n, \beta}\cneq 
\int_{[I_n(X, \beta)]^{\rm{vir}}}1\in \mathbb{Z}.$$}
\end{defi}
Recall that for any scheme $M$, Behrend~\cite{Beh} associates 
a canonical constructible 
function, 
\begin{align}\label{Behcon}
\nu_{M} \colon M \to \mathbb{Z}, 
\end{align}
such that if $M$ has a symmetric perfect obstruction theory, 
we have 
$$\int_{[M]^{\rm{vir}}}1 =\sum_{i\in \mathbb{Z}}i\chi(\nu_{M}^{-1}(i)).$$
Here $\chi(\ast)$ is the topological Euler characteristic. 
In this way, the invariant $I_{n, \beta}$ is also 
defined as a weighted Euler characteristic with respect to 
the Behrend function on $I_n(X, \beta)$. 
The relevant generating series are defined as follows. 
\begin{defi}\emph{
Let $f\colon X\to Y$ be a flopping contraction. 
We define the generating series ${}{\DT}(X)$
and ${}{\DT}(X/Y)$ to be 
\begin{align*}
{}{\DT}(X)&\cneq \sum_{n, \beta}I_{n, \beta}x^n y^{\beta}, \\
{}{\DT}(X/Y)&\cneq \sum_{n, f_{\ast}\beta=0}
I_{n, \beta}x^n y^{\beta}. 
\end{align*}
The reduced series are defined by 
\begin{align*}
\DT'(X)\cneq \frac{\DT(X)}{\DT_{0}(X)}, \quad 
\DT'(X/Y)\cneq \frac{\DT(X/Y)}{\DT_{0}(X)}.
\end{align*}
Here $\DT_{0}(X)$ is given by~\cite{BBr},
\cite{Li}, \cite{LP},
$$\DT_{0}(X)\cneq \sum_{n}I_{n, 0}x^n
=M(-x)^{\chi(X)},$$
for the MacMahon function, 
$$M(x)=\prod_{k\ge 1}(1-x^{k})^{-k}.$$}
\end{defi}

\subsection{Pandharipande-Thomas theory}
The notion of stable pairs and the associated counting 
invariants are introduced by Pandharipande and 
Thomas~\cite{PT} in order to 
give a geometric interpretation of the reduced DT theory. 
\begin{defi}\emph{\bf{\cite{PT}}}
\emph{
A 
pair $(F, s)$ is a
\textit{stable pair}
if it satisfies the following conditions.}
\begin{itemize}
\item \emph{$F\in \Coh_{\le 1}(X)$ is a pure 1-dimensional 
sheaf. }
\item \emph{$s\colon \oO_X \to F$ is a 
morphism with 0-dimensional cokernel.} 
\end{itemize}
\end{defi}
As a convention, the pair $(0, 0)$
is also a stable pair. 
For $(n, \beta)\in \mathbb{Z}\oplus N_1(X)$, 
we denote by 
$P_n(X, \beta)$ the moduli space of stable pairs 
$(F, s)$ with
$$[F]=\beta, \quad \chi(F)=n.$$
It is proved in~\cite{PT} that $P_n(X, \beta)$
is a projective scheme
with a symmetric perfect obstruction theory, by 
viewing a stable pair $(F, s)$ as a two 
term complex, 
\begin{align}\label{twoterm}
I^{\bullet}=
\cdots \to 0 \to \oO_X \stackrel{s}{\to}F \to 0 \to \cdots
\in D^b(X).
\end{align}
We also call the two term complex (\ref{twoterm}) as a 
stable pair. 
There is an associated virtual fundamental cycle, 
$$[P_n(X, \beta)]^{\rm{vir}} \in A_{0}(P_n(X, \beta)).$$
\begin{defi}\emph{
The \textit{Pandharipande-Thomas invariant} 
$P_{n, \beta}$ is defined as
$$P_{n, \beta}=\int_{[P_n(X, \beta)]^{\rm{vir}}}1 \in \mathbb{Z}.$$}
\end{defi}
The relevant generating series are defined as follows. 
\begin{defi}\emph{
Let $f\colon X\to Y$ be a flopping contraction. 
We define the generating series 
${}{\PT}(X)$ and ${}{\PT}(X/Y)$ to be}
\begin{align*}
{}{\PT}(X)&=\sum_{n, \beta}P_{n, \beta}x^n y^{\beta}, \\
{}{\PT}(X/Y)&=\sum_{n, f_{\ast}\beta=0}
P_{n, \beta}x^n y^{\beta}.
\end{align*}
\end{defi}
The following
result, which is conjectured in~\cite[Conjecture~3.3]{PT}, 
is proved in~\cite{Tcurve1}
using the results of~\cite{JS} and
the announced result in~\cite{BG}.  
\begin{thm}{\bf\cite[Theorem~1.2]{Tcurve1}}\label{mainTcurve}
Assuming Conjecture~\ref{conj:BG} below, we have 
\begin{align*}
{}{\DT}'(X)&=\PT(X), \\
{}{\DT}'(X/Y)&=\PT(X/Y).
\end{align*}
In particular, we have the equality of the generating series, 
$$\frac{{}{\DT}(X)}{{}{\DT}(X/Y)}
=\frac{{}{\PT}(X)}{{}{\PT}(X/Y)}.$$
\end{thm}

\begin{rmk}\label{rmk:BG}
The result of Theorem~\ref{mainTcurve}
is proved 
in the arXiv version of~\cite[Theorem~8.11]{Tcurve1}
under the assumption that 
Behrend-Getzler's announced result~\cite{BG}
is true. 
In Conjecture~\ref{conj:BG}, 
we formulate a required result as 
a derived category version of~\cite[Theorem~5.3]{JS}.
Since the paper~\cite{BG} has not 
yet appeared, the result of Theorem~\ref{mainTcurve}
is only proved for an Euler characteristic version
as in~\cite[Theorem~1.2]{Tcurve1} at this moment. 

Roughly speaking,
in proving DT/PT correspondence, 
 we need the result of~\cite{BG}
that locally the moduli space of semi-Schur objects 
in the derived category is the critical locus of 
some holomorphic function modulo the gauge action.
This result is needed in showing the 
derived category version of~\cite[Theorem~5.2]{JS},
that is the existence of the Lie algebra morphism 
from the Lie algebra of virtual 
indecomposable objects (cf.~\cite[Subsection~5.2]{Joy2}) in 
the Hall algebra of 
the heart of a t-structure in the derived category
to the Lie algebra defined by the Euler pairing on 
$K(X)$. In~\cite[Theorem~5.2]{JS}, 
Joyce-Song constructed the above Lie algebra morphism 
for the heart of a standard t-structure, (i.e. 
$\Coh(X)$,) by proving that the moduli 
space of objects $E\in \Coh(X)$ is written 
as a critical locus of some holomorphic 
function modulo gauge action. 
If we assume the result of~\cite{BG}, then 
the argument of~\cite{JS} is applied to 
show the derived category version of~\cite[Theorem~5.2]{JS}. 
The remaining argument is the same as in the Euler characteristic 
version. 
\end{rmk}

\subsection{Non-commutative Donaldson-Thomas theory}\label{subsec:Non}
Here we introduce (global) non-commutative 
DT invariants associated to 
an arbitrary flopping contraction $f\colon X\to Y$. 
Recall the definition of $\ppPPer(X/Y)$
in Definition~\ref{def:pp}. 
\begin{defi}\emph{
An object $I\in \ppPPer(X/Y)$ is called 
a \textit{perverse ideal sheaf} if 
there is an injection 
$I\hookrightarrow \oO_X$ in $\ppPPer(X/Y)$.}
\end{defi}
The moduli  
theory of perverse ideal sheaves
is studied by Bridgeland~\cite{Br1}. 
\begin{thm}\emph{\bf{\cite[Theorem~5.5]{Br1}}}
For $(n, \beta)\in \mathbb{Z}\oplus N_1(X)$, 
the functor of families of perverse ideal 
sheaves $I\in \ppPPer(X/Y)$ which fit into 
the exact sequence in $\ppPPer(X/Y)$, 
\begin{align}\label{PerF}
0 \lr I \lr \oO_X \lr F \lr 0,
\end{align}
satisfying 
\begin{align}\label{I(A)}
F\in \ppPPer_{\le 1}(X/Y), \quad [F]=\beta \quad \mbox{and} \quad 
\chi(F)=n,
\end{align}
is representable by a projective scheme 
$I_n(\pAA_Y, \beta)$. 
\end{thm}
\begin{rmk}
In~\cite[Theorem~5.5]{Br1}, Bridgeland 
constructs the moduli space $I_n(\pAA_Y, \beta)$
only in the case of $p=-1$. 
However the case of $p=0$ is reduced to the case of 
$p=-1$ by passing to the flop via the equivalence (\ref{standard}). 
\end{rmk}
\begin{rmk}\label{rmk:corre}
The object $F\in \ppPPer_{\le 1}(X/Y)$
in the sequence (\ref{PerF}) corresponds to an $\pAA_Y$-module 
$F'$ which admits surjections, 
\begin{align}\label{cyclic}
\pAA_{Y} \twoheadrightarrow \pAA_{Y}' \twoheadrightarrow F', 
\end{align}
in $\Coh(\pAA_Y)$ via the equivalence (\ref{restper}). 
In this way, 
 $I_n(\pAA_Y, \beta)$ is also interpreted as
a moduli space of cyclic $\pAA_Y$-modules
of a given numerical type. 
\end{rmk}
By~\cite{HT2}, there is a symmetric perfect obstruction theory on 
$I_n(\pAA_Y, \beta)$, and the associated virtual fundamental cycle, 
$$[I_n(\pAA_Y, \beta)]^{\rm{vir}} \in A_0(I_n(\pAA_Y, \beta)).$$
\begin{defi}\emph{
The \textit{(global) non-commutative Donaldson-Thomas invariant}
$\pAA_{n, \beta}$ is defined by 
$$\pAA_{n, \beta}=\int_{[I_n(\pAA_Y, \beta)]^{\rm{vir}}}1 \in \mathbb{Z}.$$}
\end{defi}
The generating series are defined as follows. 
\begin{defi}\label{def:glncDT}
\emph{
We define the generating series ${}{\DT}(\pAA_Y)$
and ${}{\DT}_0(\pAA_Y)$ to be }
\begin{align*}
{}{\DT}(\pAA_Y)&=\sum_{n, \beta}\pAA_{n, \beta}x^n y^{\beta}, \\
{}{\DT}_{0}(\pAA_Y)&=\sum_{n, f_{\ast}\beta=0}
\pAA_{n, \beta}x^n y^{\beta}.
\end{align*}
\end{defi}
\begin{rmk}
If $f\colon X\to Y$ contracts only single rational 
curve $C\subset X$ with normal bundle 
$N_{C/X}=\mathcal{O}_{C}(-1)^{\oplus 2}$, then 
the series $\DT_{0}(\pAA_Y)$ coincides with the one 
introduced by Szendr{\H o}i~\cite{Sz}
by Remark~\ref{rmk:corre}.
\end{rmk}

\section{Weak stability conditions on 
$\dD_X$}\label{sec:weak}
In what follows, we use the notation introduced
in the previous section. 
Let $X$ be a smooth projective Calabi-Yau 
3-fold with a flopping contraction
$f\colon X\to Y$. (cf.~Definition~\ref{def:flopping}.)
In this section, we study the space of weak stability 
conditions on the triangulated subcategory, 
$$\dD_X\cneq \langle \oO_X, \Coh_{\le 1}(X) \rangle_{\tr}
\subset D^b(\Coh(X)).$$
We set $\Gamma$ to be 
$$\Gamma=\mathbb{Z}\oplus N_1(X) \oplus \mathbb{Z}, $$
and a group homomorphism $\cl \colon K(\dD_X) \to \Gamma$
to be 
$$\cl(E)=(\ch_3(E), \ch_2(E), \ch_0(E)).$$
By the definition of $\dD_X$, it is obvious that 
$\ch_{\bullet}(E)$ has integer coefficients 
for $E\in \dD_X$, thus $\cl$ is well-defined. 
\begin{rmk}
Let $I_C \subset \oO_X$ be an ideal sheaf of
a 1-dimensional subscheme $C\subset X$. 
We have $I_C \in \dD_X$, and 
$$\cl(I_C)=(-n, -\beta, 1) \mbox{ if and only if }
[C]=\beta, \chi(\oO_C)=n, $$
by Riemann-Roch theorem. 
The similar statement also holds for 
stable pairs (\ref{twoterm}) and perverse ideal sheaves. 
\end{rmk}
We denote by $\rk$ the projection onto the third
factor, 
$$\rk \colon \Gamma\ni (s, l, r) \mapsto r \in \mathbb{Z}.$$
We set 
\begin{align*}
\Gamma_0 &=\mathbb{Z} \oplus N_1(X/Y), \\
\Gamma_1 &=\mathbb{Z} \oplus N_1(X).
\end{align*}
We have the filtration, 
\begin{align}\label{filt2}
\Gamma_0 \stackrel{i}{\hookrightarrow} \Gamma_1 
\stackrel{j}{\hookrightarrow} \Gamma_2=\Gamma, 
\end{align}
via $i(s, l)=(s, l)$ and $j(s, l)=(s, l, 0)$. 
Each subquotient $\mathbb{H}_i=\Gamma_i/\Gamma_{i-1}$
is 
$$\mathbb{H}_0=\mathbb{Z}\oplus N_1(X/Y), \quad 
\mathbb{H}_1=N_1(Y), \quad \mathbb{H}_2=\mathbb{Z},$$
and there is a natural isomorphism, 
\begin{align}\label{natiso}
\left(
\mathbb{C}\times N^1(X/Y)_{\mathbb{C}} \right)
\times N^1(Y)_{{\mathbb{C} }} \times \mathbb{C}
\stackrel{\sim}{\lr}
\prod_{i=0}^{2}\mathbb{H}_i^{\vee}.
\end{align}
Hence we have the local 
homeomorphism by Theorem~\ref{thm:stab},  
\begin{align}\label{loc:hom}
\Pi \colon 
\Stab_{\Gamma_{\bullet}}(\dD_X) \to
\left(
\mathbb{C}\times N^1(X/Y)_{\mathbb{C}} \right)
\times N^1(Y)_{\mathbb{C}} \times \mathbb{C}.
\end{align}
The goal of this section is to construct
a certain region, 
\begin{align*}
\uU_X \subset \Stab_{\Gamma_{\bullet}}(\dD_X), 
\end{align*}
and show that the above region 
can be patched by the derived equivalence under 
a flop.

\subsection{t-structures on $\dD_X$}
In this subsection, we construct
t-structures on $\dD_X$. 
The notation used here 
is introduced in subsection~\ref{subsec:der}. 
\begin{lem}\label{lem:AB}
(i) There is the heart of a bounded t-structure
$\aA_X \subset \dD_X$, written as 
\begin{align}\label{writeA}
\aA_X=\langle \oO_X, \Coh_{\le 1}(X)[-1]\rangle_{\ex}.
\end{align}
(ii) There are hearts of bounded t-structures 
$\pB_{X/Y}\subset \dD_{X}$ for $p=0, -1$, written as 
\begin{align}\label{pB}
\pB_{X/Y}=\langle \oO_X, \ppPPer_{\le 1}(X/Y)[-1]\rangle_{\ex}.
\end{align}
\end{lem}
\begin{proof}
(i) is proved in~\cite[Lemma~3.5]{Tcurve1}, so we prove (ii). 
Let us consider the heart of a bounded t-structure 
$\ppPPer^{\dag}(X/Y)\subset D^b(\Coh(X))$, 
given in Definition~\ref{def:perdag}.
 By the construction, it is obvious that 
$$\ppPPer^{\dag}(X/Y)[-1] \cap D^b(\Coh_{\le 1}(X))=
\ppPPer_{\le 1}(X/Y)[-1].$$
 Take $F\in \ppPPer_{\le 1}(X/Y)$. Since
 $\oO_X \in \ppPPer(X/Y)$, we have 
 $\Hom(\oO_X, F[-1])=0$. Also we have 
 \begin{align*}
 \Hom(F[i], \oO_X) &\cong \Hom(\oO_X, F[3+i])^{\vee} \\
 &\cong \Hom(\oO_Y, \dR f_{\ast}F[3+i])^{\vee} \\
 &= 0,
 \end{align*}
 for $i\ge -1$. 
 Here the first isomorphism follows from the Serre duality, 
 the second one is an adjunction, and the last 
one is a consequence of $\dR f_{\ast}F \in \Coh(Y)$
by the definition of $\ppPPer(X/Y)$. 
In particular, we have 
$$\oO_X \in \ppPPer_{\ge 2}(X/Y) \subset \ppPPer^{\dag}(X/Y)[-1].$$
 Applying Proposition~\ref{t-str} below 
 by setting $\dD=D^b(\Coh(X))$, $\dD'=D^b(\Coh_{\le 1}(X))$,
 $\aA=\ppPPer^{\dag}(X/Y)[-1]$
 and $E=\oO_X$, we obtain the result.  
\end{proof}
We have used the following proposition, 
which is proved in~\cite[Proposition~3.6]{Tcurve1}. 
\begin{prop}\emph{\bf{\cite[Proposition~3.6]{Tcurve1}}}
\label{t-str}
Let $\dD$ be a $\mathbb{C}$-linear triangulated category 
and $\aA \subset \dD$ the heart of a bounded t-structure on $\dD$. 
Take $E\in \aA$ with
$\End(E)=\mathbb{C}$ and a full triangulated subcategory
 $\dD'\subset \dD$, which satisfy 
 the following conditions. 
\begin{itemize}
 \item The category $\aA'\cneq \aA\cap \dD'$
is the heart of a bounded t-structure on $\dD'$, which 
is closed under subobjects and quotients in the 
abelian category $\aA$. 
\item For any object $F\in \aA'$, we have 
\begin{align}\label{cond}
\Hom(E, F)=\Hom(F, E)=0.
\end{align}
\end{itemize}
Let $\dD_E$ be the triangulated category, 
$$\dD_E \cneq \langle E, \dD' \rangle_{\tr} \subset \dD.$$
Then $\aA_E\cneq \dD_E \cap \aA$ is the heart 
of a bounded t-structure on $\dD_E$, 
which satisfies 
$$\aA_E=\langle E, \aA' \rangle_{\ex}.$$
\end{prop}
\begin{rmk}\label{rmk:B}
By Lemma~\ref{lem:straight} and (\ref{pB}), 
the abelian categories $\pB_{X/Y}$ are written as 
\begin{align}\label{per:gen10}
\oB_{X/Y}&=
\langle \oO_X, 
\omega_{f^{-1}(y)}, \oO_{C_i}(-1)[-1], \widetilde{\Coh}_{\le 1}(X)[-1]
\rangle_{\ex}, \\ 
\label{per:gen20}
\iB_{X/Y}&=
\langle \oO_X, \oO_{f^{-1}(y)}[-1], \oO_{C_i}(-1), 
\widetilde{\Coh}_{\le 1}(X)[-1]
\rangle_{\ex}. 
\end{align}
\end{rmk}
We have the following lemma, 
whose proof will be given in Section~\ref{sec:tech}. 
\begin{lem}\label{lem:noether}
(i) The abelian categories $\aA_X$, $\pB_{X/Y}$ are noetherian.

(ii)  Any infinite chain of monomorphisms in $\aA_X$, 
(resp.~$\pB_{X/Y}$,)
\begin{align}\label{stmo2}
E_0 \hookleftarrow E_1 \hookleftarrow \cdots E_j
\hookleftarrow E_{j+1} \hookleftarrow \cdots, 
\end{align}
with $E_j/E_{j+1}\notin \Coh_{0}(X)[-1]$, 
(resp.~$E_j/E_{j+1}\notin \ppPPer_{0}(X/Y)[-1]$,) terminates. 
\end{lem}
Let us see that $\pB_{X/Y}$ is obtained 
from $\aA_{X}$ via tilting. 
Let $\pF$ for $p=0, -1$ be 
\begin{align*}
\oF&\cneq \{ F\in \Coh_{\le 1}(X) \mid f_{\ast}F=0, 
\ \Hom(\cC, F)=0\}, \\
\iF &\cneq \{ F\in \Coh_{\le 1}(X) \mid f_{\ast}F=0\}.
\end{align*}
(See Definition~\ref{def:pp} for $\cC\subset \Coh(X)$.)
Then $\pF$ fit into torsion 
pairs 
\begin{align}\label{pTpF}
(\pT, \pF), 
\end{align} on $\Coh_{\le 1}(X)$
such that 
$\ppPPer_{\le 1}(X/Y)$ is the associated 
tilting. (cf.~\cite[Section~3]{MVB}.)
Also note that 
$\Hom(E, F)=0$ for $E \in \Coh_0(X)[-1]$
and $F \in \pF[-1]$, which follows
from the definition of 
$\pF$.
Therefore by Lemma~\ref{lem:later} and 
Lemma~\ref{lem:noether}, 
the subcategories $\pF[-1]\subset \aA_{X}$ also
fit into torsion pairs on $\aA_{X}$, denoted by 
\begin{align}\label{pT}
(\pT', \pF[-1]).
\end{align}
\begin{lem}\label{lem:tiltAB}
The abelian category $\pB_{X/Y}$ is the tilting 
with respect to $(\pT', \pF[-1])$, i.e. 
\begin{align}\label{tilt}
\pB_{X/Y} = \langle \pF, \pT' \rangle_{\ex}.
\end{align}
\end{lem}
\begin{proof}
We show the case of $p=0$. 
It is enough to show 
that the LHS of (\ref{tilt}) is contained 
in the RHS of (\ref{tilt}),
since both are hearts of bounded t-structures on 
$\dD_X$. 
By Remark~\ref{rmk:B}, any object in the 
 LHS of (\ref{tilt}) is given by 
 a successive extensions of objects
$\oO_X$, $\omega_{f^{-1}(y)}$ for $y\in \Sing(Y)$, 
$\oO_{C_i}(-1)[-1]$ and objects in $\widetilde{\Coh}_{\le 1}(X)[-1]$. 
Thus it suffices to show that these objects 
are contained in the RHS of (\ref{tilt}). 
We have 
\begin{align}
\label{enough1}
&\Hom(\oO_X, \oF[-1])=0, \quad \Rightarrow \quad 
\oO_X \in \oT', \\
\label{enough2}
&\omega_{f^{-1}(y)}[1]\in \oPPer_{\le 1}(X/Y), \quad \Rightarrow \quad 
\omega_{f^{-1}(y)} \in \oF \\
\label{enough3}
&\Hom(\oO_{C_i}(-1)[-1], \oF[-1])=0, \quad \Rightarrow 
\quad \oO_{C_i}(-1)[-1] \in \oT', \\
\label{enough4}
&\Hom(\widetilde{\Coh}_{\le 1}(X)[-1], \oF[-1])=0, 
\quad \Rightarrow \quad \widetilde{\Coh}_{\le 1}(X)[-1]\subset 
\oT'.
\end{align}
Here (\ref{enough3}) follows from the 
definition of $\oF$, and (\ref{enough4}) follows 
from $f_{\ast}F=0$ for any $F\in \oF$. 
Hence (\ref{tilt}) holds. 
\end{proof}

\subsection{Constructions of weak stability conditions
(neighborhood of the large volume limit)}
Here we construct weak stability conditions on $\dD_X$, 
whose corresponding heart of bounded t-structure is 
$\aA_X$. (cf.~Lemma~\ref{lem:AB}.)
The set of weak stability conditions constructed here is 
interpreted as a neighborhood of the large volume limit 
at $X$
in terms of string theory. 
Let us take the elements, 
$$B+i\omega \in A(X/Y)_{\mathbb{C}}, \quad 
\omega'\in A(Y), \quad z\in \mathfrak{H} \mbox{ with }
\arg z \in (\pi/2, \pi). $$
The data 
\begin{align}\label{data1}
\xi=(1, -(B+i\omega), -i\omega', z),
\end{align}
in the LHS of (\ref{natiso}) 
determines the element $Z_{\xi}\in \prod_{i=0}^{2}\mathbb{H}_i^{\vee}$
via the isomorphism (\ref{natiso}). 
It is written as 
\begin{align*}
Z_{0, \xi}&\colon \mathbb{Z} \oplus N_1(X/Y) \ni (s, l)
\mapsto s-(B+i\omega)l, \\
Z_{1, \xi}&\colon N_1(Y) \ni l' \mapsto -i\omega'\cdot l', \\
Z_{2, \xi}&\colon \mathbb{Z} \ni r \mapsto zr.
\end{align*}
\begin{lem}\label{lem:pair1}
The pairs 
\begin{align}\label{pair4}
\sigma_{\xi}=(Z_{\xi}, \aA_X), 
\quad \xi \mbox{ is given by }(\ref{data1}),
\end{align} determine 
points in $\Stab_{\Gamma_{\bullet}}(\dD_X)$. 
\end{lem}
\begin{proof}
We check that (\ref{def:weak}) holds for any non-zero $E\in \aA_X$. 
We write $\cl(E)=(-n, -\beta, r)$
for $n\in \mathbb{Z}$, 
$\beta \in N_1(X)$ and $r\in \mathbb{Z}$. 
By (\ref{writeA}), we have one of the following. 
\begin{itemize}
\item We have $r>0$. In this case, we have 
$$Z_{\xi}(E)=zr \in \mathfrak{H}.$$
\item We have $r=0$, $\beta\in \NE(X)$ and $f_{\ast}\beta\neq 0$. 
In this case, we have
$$Z_{\xi}(E)=i\omega' \cdot f_{\ast}\beta \in \mathfrak{H}.$$
\item We have $r=0$ and $\beta \in \NE(X/Y)$. In this case, 
we have 
$$Z_{\xi}(E)=-n+(B+i\omega)\beta \in \mathfrak{H}.$$
\end{itemize}
The proofs to check other properties, i.e. 
Harder-Narasimhan property, support property and 
local finiteness will 
be given in 
Section~\ref{sec:tech}. 
\end{proof}
We define the subspace
$\uU_X \subset \Stab_{\Gamma_{\bullet}}(\dD_X)$
as follows. 
\begin{defi}\label{def:UX}
\emph{
We define $\uU_X \subset \Stab_{\Gamma_{\bullet}}(\dD_X)$
to be 
$$\uU_X \cneq \{ \sigma_{\xi} \colon \sigma_{\xi} \mbox{\rm{ is given by }}
(\ref{pair4})\}.$$
For a fixed $B_0\in N^1(X/Y)$, we set 
$$\uU_{X, B_0}\cneq \{ \sigma_{\xi} \in \uU_X :
\xi \mbox{\rm{ is given by }}(\ref{pair4}) \mbox{ with }
B=B_0\}.$$}
\end{defi}
By Lemma~\ref{lem:later}, the map $\xi \mapsto \sigma_{\xi}$
is continuous. 
In particular $\uU_X$ and $\uU_{X, B_{0}}$ are
 connected subspaces. 
The map (\ref{loc:hom}) restricts to the homeomorphisms, 
\begin{align}\label{Uhom}
&\Pi \colon \uU_{X} \stackrel{\sim}{\to}
\{1\} \times \{-A(X/Y)_{\mathbb{C}}\} \times 
\{-iA(Y)\} \times \mathfrak{H}', \\
\notag
&\Pi \colon \uU_{X, B_{0}} \stackrel{\sim}{\to}
\{1\} \times \{-\left(B_{0}+iA(X/Y)\right)\} \times 
\{-iA(Y)\} \times \mathfrak{H}'.
\end{align}
where $\mathfrak{H}'$ is 
$$\mathfrak{H}'=\{ z\in \mathfrak{H} :
\arg z\in (\pi/2, \pi)\}.$$
 \begin{rmk}
 The subspace $\uU_X \subset \Stab_{\Gamma_{\bullet}}(\dD_X)$
 is interpreted as a kind of limiting degeneration of 
 the 
 neighborhood of the large volume limit in string theory. 
 In fact 
 for $B+i\omega \in A(X)_{\mathbb{C}}$, let
  $Z_{(B, \omega)}\colon K(X) \to \mathbb{C}$ be 
 $$Z_{(B, \omega)}(E)=
 \int e^{-(B+i\omega)}\ch(E)\sqrt{\td_X} \in \mathbb{C}.$$
 If $E\in \Coh_{\le 1}(X)[-1]$ with $\cl(E)=(-n, -\beta, 0)$, 
 we have 
 $$Z_{(B, \omega)}(E)=-n+(B+i\omega)\beta, $$
 which coincides with $Z_{0, \xi}(\cl(E))$. 
 \end{rmk}

\subsection{Construction of weak stability conditions
(non-commutative points)}
Here we construct another family of weak stability conditions, 
whose corresponding hearts of bounded t-structures 
are $\pB_{X/Y}$.
(cf.~Lemma~\ref{lem:AB}.) Let $C_1, \cdots, C_N$ be the 
irreducible components of the exceptional locus of 
a flopping contraction $f\colon X\to Y$. 
We denote by $Z_y$ the fundamental cycle of the 
scheme theoretic fiber of $f$ at $y\in \Sing(Y)$. 
For $p=0, -1$, we set 
$\pV(X/Y)$ as follows, 
\begin{align}\label{pV}
\pV(X/Y)\cneq \left\{
B\in N^1(X/Y)_{\mathbb{R}} : 
\begin{array}{l}
(-1)^{p}B\cdot C_i <0, \ (-1)^{p}B\cdot Z_y >-1, \\
\mbox{ for all }1\le i\le N \mbox{ and }y\in \Sing(Y) \end{array}
\right\}.
\end{align}
For the elements, 
\begin{align}\notag
&B\in \pV(X/Y), \quad \omega'\in A(Y), \\
\label{element2}
& z_0, z_1 \in 
\mathfrak{H} \mbox{ with } \arg z_i \in (\pi/2, \pi], \ z_1\neq -1,
\end{align}
the data 
\begin{align}\label{data2}
\xi=(-z_0, z_0 B, -i\omega', z_1), 
\end{align}
in the LHS of (\ref{natiso}) 
determines the element $Z_{\xi}\in \prod_{i=0}^{2}\mathbb{H}_i^{\vee}$
via the isomorphism (\ref{natiso}). 
It is written as 
\begin{align*}
Z_{0, \xi}&\colon \mathbb{Z} \oplus N_1(X/Y) \ni (s, l)
\mapsto z_0(-s+Bl), \\
Z_{1, \xi}&\colon N_1(Y) \ni l' \mapsto -i\omega'\cdot l', \\
Z_{2, \xi}&\colon \mathbb{Z} \ni r \mapsto z_1r.
\end{align*}
\begin{lem}\label{lem:constpair2}
The pairs 
\begin{align}\label{const:pair2}
\sigma_{\xi}=(Z_{\xi}, \pB_{X/Y}), \quad 
\xi \mbox{ is given by }(\ref{data2}),
\end{align}
determine points in $\Stab_{\Gamma_{\bullet}}(\dD_X)$. 
\end{lem}
\begin{proof}
For simplicity we show the case of $p=0$. 
In order to check (\ref{def:weak}), it is 
enough to show this for generators of $\oB_{X/Y}$, 
given in (\ref{per:gen10}). 
We have 
\begin{align*}
Z_{2, \xi}(\oO_X)&=z_1 \in \mathfrak{H}, \\
Z_{1, \xi}(F[-1])&=i\omega' \cdot f_{\ast}\beta \in \mathfrak{H}, \\
Z_{0, \xi}(\oO_{C_i}(-1)[-1])&=-z_0B \cdot C_i \in \mathfrak{H}, \\
Z_{0, \xi}(\omega_{f^{-1}(y)})&=z_0(1+B\cdot Z_y) \in \mathfrak{H}, 
\end{align*}
by our choice of $z_i$ and $B$. Here 
$0\neq F\in \widetilde{\Coh}_{\le 1}(X)$ 
satisfies $\cl(F)=(n, \beta, 0)$. 
Note that $f_{\ast}\beta\in N_1(Y)$
is a non-zero effective 
class by the definition of $\widetilde{\Coh}_{\le 1}(X)$. 
Therefore 
the pair $(Z_{\xi}, \pB_{X/Y})$
satisfies (\ref{def:weak}). 
The Harder-Narasimhan property, the local finiteness 
and the support property are proved along with 
the same argument of Lemma~\ref{lem:pair1}, and  
we leave the readers to check the detail. 
\end{proof}
We define the subspaces 
$\pmV_{X/Y}\subset 
\pU_{X/Y} \subset \Stab_{\Gamma_{\bullet}}(\dD_{X})$
as follows. 
\begin{defi}\label{def:pU}
\emph{
We define $\pU_{X/Y}$, $\pmV_{X/Y}$ to be}
\begin{align*}
\pU_{X/Y}&=\{\sigma_{\xi} : 
\sigma_{\xi} \mbox{ \rm{is given by} }(\ref{const:pair2})\}, \\
\pmV_{X/Y}&=\{\sigma_{\xi}\in \pU_{X/Y} :
\xi \mbox{ \rm{is given by} }(\ref{data2}) \mbox{ \rm{with} }
z_0=-1\}.
\end{align*}
\end{defi}
By Lemma~\ref{lem:later}, the subspaces 
$\pU_{X/Y}$ and $\pmV_{X/Y}$ are connected. 
The map (\ref{loc:hom}) restricts to the homeomorphism, 
\begin{align}\label{lochom2}
\Pi \colon \pmV_{X/Y} \stackrel{\sim}{\to}
\{1\} \times \{-\pV(X/Y)\} \times 
\{-iA(Y)\} \times \mathfrak{H}'.
\end{align}

\subsection{Flops and weak stability conditions}
Let $\phi \colon X^{+} \dashrightarrow X$ be 
a flop as in the diagram (\ref{fig:flop}), 
and $\Phi$ a standard equivalence given in 
(\ref{standard}). 
Since the kernel of $\Phi$ is supported on the fiber
product $X\times _Y X^{+}$, (cf.~\cite[Proposition~4.2]{Ch},)
$\Phi$ restricts to the equivalence
\begin{align}\label{Dstan}
\Phi \colon \dD_{X^{+}} \stackrel{\sim}{\lr} \dD_X.
\end{align}
\begin{lem}\label{stanrest}
The standard equivalence $\Phi\colon 
\dD_{X^{+}} \to \dD_{X}$ restricts to the equivalence
between $\iB_{X^{+}/Y}$ and $\oB_{X/Y}$. 
\end{lem}
\begin{proof}
Note that $\Phi$ induces the  
equivalences between $\iPPer(X^{+}/Y)$
and $\oPPer(X/Y)$, and 
the equivalence 
between 
 $\iPPer_{\le 1}(X^{+}/Y)$ and 
 $\oPPer_{\le 1}(X/Y)$. Hence
  $\Phi$ induces the equivalence, 
 \begin{align}\label{PhiA}
 \Phi \colon \iPPer^{\dag}(X^{+}/Y)[-1] \stackrel{\sim}{\lr}
 \oPPer^{\dag}(X/Y)[-1],
 \end{align}
 where $\pPPer^{\dag}(X/Y)$ is given in 
 Definition~\ref{def:perdag}. 
 Since we have 
 $$\pB_{X/Y}=\dD_X \cap \ppPPer^{\dag}(X/Y)[-1], $$
 by Proposition~\ref{t-str}, 
 we obtain the result by restricting
 (\ref{PhiA}) to $\dD_{X^{+}}$ and $\dD_{X}$. 
 \end{proof}
Similar to $\Gamma_{\bullet}$, we set 
\begin{align*}
&\Gamma_0^{+}=\mathbb{Z}\oplus N_1(X^{+}/Y), \\
&\Gamma_1^{+}=\mathbb{Z}\oplus N_1(X^{+}), \\
&\Gamma^{+}=\Gamma_2^{+}=\mathbb{Z} \oplus 
 N_{1}(X^{+}) \oplus \mathbb{Z}. 
\end{align*}
The associated subquotient is denoted by 
$\mathbb{H}^{+}_i$. 
\begin{lem}\label{preiso}
There is a filtration preserving 
isomorphism 
$\Phi_{\Gamma}\colon \Gamma^{+}_{\bullet} \to
\Gamma_{\bullet}$, which satisfies the 
following. 
\begin{itemize}
\item
The following diagram commutes,
\begin{align}\label{comdi0}
\xymatrix{
\dD_{X^{+}} \ar[r]^{\Phi} \ar[d]_{\cl} & \dD_{X} \ar[d]^{\cl} \\
\Gamma^{+}\ar[r]^{\Phi_{\Gamma}} & \Gamma.
}
\end{align}
\item The induced morphism $\gr_{\bullet} \Phi_{\Gamma}$ satisfies
\begin{align*}
\gr_0 \Phi_{\Gamma}
&\colon \mathbb{Z} \oplus N_1(X^{+}/Y)
 \ni (z, C) \\
 & \qquad  \mapsto (z, \phi_{\ast}C)
 \in \mathbb{Z} \oplus N_1(X/Y) \\
 \gr_1 \Phi_{\Gamma}& \colon N_1(Y) \ni D' \mapsto 
 D' \in N_1(Y), \\
 \gr_2  \Phi_{\Gamma} &\colon \mathbb{Z} \ni z' \mapsto 
 z' \in \mathbb{Z}.
\end{align*}
\end{itemize}
\end{lem}
\begin{proof}
First the following diagram is commutative
by~\cite[Proposition~5.2]{ToBPS}, 
\begin{align}\label{comdi}\xymatrix{
D^b(\Coh_{\le 1}(X^{+})) \ar[r]^{\Phi} \ar[d]_{(\ch_3, \ch_2)}
 & D^b(\Coh_{\le 1}(X)) \ar[d]^{(\ch_3, \ch_2)} \\
\mathbb{Z}\oplus 
N_{1}(X^{+})\ar[r]^{(\id, \phi_{\ast})} & \mathbb{Z}\oplus N_{1}(X).
}\end{align}
Let $v=\cl \Phi(\oO_{X^{+}}) \in \Gamma$
and set $\Phi_{\Gamma}$ as 
$$\Phi_{\Gamma}(s, l, r) = (s, \phi_{\ast}l, 0)+rv.$$
The commutativity of (\ref{comdi}) 
implies that $\Phi_{\Gamma}$ fits into 
the commutative diagram (\ref{comdi0}). 
Since $v$ is of the form $(\ast, \ast, 1)$, 
the map $\Phi_{\Gamma}$ is isomorphism.
The induced isomorphism $\gr_{\bullet}\Phi_{\Gamma}$
is of the desired form by the 
construction. 
\end{proof}
\begin{rmk}
In fact one can show that $\Phi(\oO_{X^{+}})\cong \oO_{X}$, 
hence $\cl \Phi(\oO_{X^{+}})=(0, 0, 1)$. 
However we do not use this fact. 
\end{rmk}
By Lemma~\ref{lem:group}
and Lemma~\ref{preiso}, we 
have the commutative diagram,  
\begin{align}\label{comdiag}
\xymatrix{
\Stab_{\Gamma^{+}_{\bullet}}(\dD_{X^{+}})
\ar[r]^{\Phi_{\ast}} \ar[d]_{\Pi^{+}}
& \Stab_{\Gamma_{\bullet}}(\dD_{X}) \ar[d]^{\Pi} \\
\prod_{i=0}^2 \mathbb{H}_i^{+ \vee} \ar[r]^{(\gr\Phi_{\Gamma}^{-1})^{\vee}} &
\prod_{i=0}^2 \mathbb{H}_i^{\vee}.
}\end{align}
\begin{prop}\label{VU}
(i) We have 
\begin{align}\label{PhiV}
\Phi_{\ast}(\imV_{X^{+}/Y})=\omV_{X/Y}.
\end{align}

(ii) We have 
$\pmV_{X/Y} \subset \overline{\uU}_{X}$. 
In particular, we have the inclusion 
\begin{align}\label{PhiV2}
\omV_{X/Y} \subset \overline{\uU}_{X} \cap 
\Phi_{\ast}\overline{\uU}_{X^{+}}, 
\end{align}
and the following subset $\uU \subset \Stab_{\Gamma_{\bullet}}(\dD_X)$
is connected, 
\begin{align}\label{connU}
\uU\cneq \uU_X \cup \Phi_{\ast}\uU_{X^{+}} \cup \oU_{X/Y} 
\cup \iU_{X/Y}.
\end{align}
\end{prop}
\begin{proof}
(i)
First note that the strict transform 
$\phi_{\ast}\colon N^1(X^{+}/Y) \to N^1(X/Y)$ 
induces the homeomorphism, 
$$\phi_{\ast}\colon \iV(X^{+}/Y) \stackrel{\sim}{\to}
\oV(X/Y).$$
Hence by the homeomorphism (\ref{lochom2}), 
the map $(\gr \Phi_{\Gamma}^{-1})^{\vee}$
in the diagram (\ref{comdiag})
induces the homeomorphism, 
$$\Pi^{+}\left(\imV_{X^{+}/Y}\right) \stackrel{\sim}{\lr} 
\Pi \left(\omV_{X/Y}\right).$$
Combined with Lemma~\ref{stanrest}, we obtain (\ref{PhiV}). 
 
 (ii) By (\ref{Uhom}) and (\ref{lochom2}), we have 
 $$\Pi\left(\pmV_{X/Y}\right) \subset
  \overline{\Pi\left(\uU_{X}\right)}.$$
 By Lemma~\ref{lem:scont}
 and Lemma~\ref{lem:tiltAB}, we obtain 
  $\pmV_{X/Y} \subset \overline{\uU}_{X}$. 
  Combined with (i), we conclude (\ref{PhiV2})
  and the connectedness of (\ref{connU}). 
\end{proof}
\begin{rmk}
The subspace
$$\overline{\uU}_{X} \cup
\Phi_{\ast}\overline{\uU}_{X^{+}} \subset \Stab_{\Gamma_{\bullet}}(\dD_X),$$
consists of two chambers $\uU_{X}$ and
$\Phi_{\ast}\uU_{X^{+}}$, which is 
an analogue of
the chamber structure on the space of 
stability conditions on $D^b(\Coh_{\le 1}(X))$. 
(cf.~\cite[Theorem~4.11]{Tst}.)
The chamber $\uU_{X}$
(resp.~$\Phi_{\ast}\uU_{X^{+}}$)
 corresponds to the neighborhood 
of the large volume limit at $X$,
(resp.~$X^{+}$,)
and the wall $\omV_{X/Y}$ 
corresponds to the locus of non-commutative points. 
See~\cite[Figure~8]{Sz}.
\end{rmk}

\section{Wall-crossing formula}\label{sec:semi}
In this section, we review the main results of~\cite[Section~5]{Tcurve1}.
As in the previous section, $f\colon X\to Y$ is a 
flopping contraction from a smooth projective Calabi-Yau 
3-fold $X$. 
\subsection{Assumption}
Here we recall the wall-crossing formula 
of generating series of Donaldson-Thomas type invariants 
under change of weak stability 
conditions, given in~\cite[Section~5]{Tcurve1}. 
The formula is established under some conditions 
given in Assumption~\ref{assum} below. 
Let us recall that, 
by the result of Lieblich~\cite{LIE},  
there is an algebraic stack $\mM$
locally of finite type over $\mathbb{C}$ which
parameterizes $E\in D^b(\Coh(X))$ satisfying 
\begin{align}\label{Ext}
\Ext^i(E, E)=0, \quad \mbox{ for any }i<0.
\end{align}
Let $\aA \subset \dD_{X}$ be the heart of a bounded 
t-structure on $\dD_X$. 
We can consider the following (abstract) substack,
$$\oO bj(\aA) \subset \mM,$$
which parameterizes objects $E\in \aA$.
The above stack decomposes as 
\begin{align*}
\oO bj(\aA)=\coprod _{v\in \Gamma} \oO bj^{v}(\aA),
\end{align*}
where $\oO bj^{v}(\aA)$
is the stack of 
objects $E\in \aA$ with $\cl(E)=v$. 

Let $\Gamma_{\bullet}$ be the filtration (\ref{filt2}). 
The wall-crossing formula~\cite[Section~5]{Tcurve1}
is applied for a certain connected subset 
$$\vV \subset \Stab_{\Gamma_{\bullet}}(\dD_X),$$
satisfying the following assumption. 
\begin{assum}\label{assum}\emph{\bf{\cite[Assumption~4.1]{Tcurve1}}}
For any $\sigma=(Z=\{Z_i\}_{i=0}^{2}, \pP) \in \vV$ with 
$\aA=\pP((0, 1])$, the following conditions 
are satisfied. 
\begin{itemize}
\item We have 
\begin{align}\label{phase1}
\oO_X \in \pP(\psi), \quad 
\frac{1}{2}<\psi < 1, 
\end{align}
and $\oO_X$ is the only object $E\in \pP(\psi)$
with $\cl(E)=(0, 0, 1)$. 
\item 
We have 
\begin{align}
Z_1(\mathbb{H}_1) \subset \mathbb{R}\cdot i.
\end{align}
\item For any $v, v'\in \Gamma_0$ and 
any other point $\tau=(W, \qQ) \in \vV$, 
 we have 
\begin{align}\label{extra}
Z(v)\in \mathbb{R}_{>0}Z(v') \quad 
\mbox{if and only if} \quad 
W(v) \in \mathbb{R}_{>0}W(v').
\end{align}
\item For any $v\in \Gamma$ with $\rk(v)= 1$
or $v\in \Gamma_0$, the stack 
of objects 
$$\oO bj^{v}(\aA) \subset \mM,$$
is an open substack of $\mM$. In particular, 
$\oO bj^{v}(\aA)$ is an algebraic stack 
locally of finite type over $\mathbb{C}$. 
\item For any $v\in \Gamma$ with $\rk(v)=1$ or 
$v\in \Gamma_{0}$, the stack of 
$\sigma$-semistable objects $E\in \aA$
with $\cl(E)=v$, 
$$\mM^v(\sigma)\subset \oO bj^{v}(\aA), $$
is an open substack of finite type 
over $\mathbb{C}$. 
\item There 
are subsets $0\in T\subset S\subset \mathbb{Z}\oplus N_1(X)$, 
which satisfy Assumption~\ref{assum2} below.  
\item For any other point $\tau \in \vV$, 
there is a good path 
(see Definition~\ref{def:good} below) in $\vV$ which connects 
$\sigma$ and $\tau$. 
\end{itemize}
\end{assum}
As for the last condition of Assumption~\ref{assum}, 
the notion of good path is defined as follows. 
\begin{defi}\label{def:good}
\emph{
A path $[0, 1] \ni t \mapsto \sigma_t \in \vV$
is \textit{good} 
if for any $t\in (0, 1)$ and 
$v\in \Gamma_{0}$ satisfying 
$Z_t(v) \in \mathbb{R}_{>0}Z_t(\oO_X)$, we have either
\begin{align}
\label{good2}
&\arg Z_{t+\varepsilon}(v)< \arg Z_{t+\varepsilon}(\oO_X), \quad 
\arg Z_{t-\varepsilon}(v)> \arg Z_{t-\varepsilon}(\oO_X), \ \mbox{or} \\
\label{good1}
&\arg Z_{t+\varepsilon}(v)> \arg Z_{t+\varepsilon}(\oO_X), \quad 
\arg Z_{t-\varepsilon}(v)< \arg Z_{t-\varepsilon}(\oO_X), 
\end{align}
for $0<\varepsilon \ll 1$.}
\end{defi}
As we mentioned in Remark~\ref{rmk:BG}, 
we assume the following conjecture.
\begin{conj}\label{conj:BG}
For any $[E] \in \mM$, 
let $G$ be a maximal reductive 
subgroup in $\Aut(E)$.
Then there exists a $G$-invariant analytic 
open neighborhood $V$ of $0$ in 
$\Ext^1(E, E)$, 
a $G$-invariant holomorphic function $f\colon V\to \mathbb{C}$
with $f(0)=df|_{0}=0$, and a smooth morphism 
of complex analytic stacks
\begin{align*}
\Phi \colon [\{df=0\}/G] \to \mM,
\end{align*}
of relative dimension $\dim \Aut(E)- \dim G$. 
\end{conj}
The above conjecture is
proved in~\cite[Theorem~5.5]{JS} if $E\in \Coh(X)$, 
and a similar result is 
announced by Behrend-Getzler~\cite{BG}. 
In what follows, we assume the above conjecture.

\subsection{Wall-crossing formula of 
the generating series}
Let $\vV \subset \Stab_{\Gamma_{\bullet}}(\dD_X)$ be 
a connected subset satisfying Assumption~\ref{assum}. 
We introduce the following notion. 
\begin{defi}\emph{
We say $\sigma =(Z, \pP) \in \vV$ is \textit{general}
if there is no $v\in \Gamma_0$ which satisfies 
$Z(v) \in \mathbb{R}_{>0}Z(\oO_X)$.} 
\end{defi}
For general $\sigma, \tau \in \vV$, take 
a good path, (cf.~Definition~\ref{def:good},) 
$$[0, 1] \ni t \mapsto \sigma_t =(Z_t, \pP_t)
\in \vV, $$
which satisfies $\sigma_0=\sigma$ and $\sigma_1=\tau$. 
For $t\in [0, 1]$, let
$W_t$ be the set, 
$$W_t=\{v\in \Gamma_{0} :
Z_{t}(v)\in \mathbb{R}_{>0}Z_{t}(\oO_X)\}.$$  
For $t\in [0, 1]$ with $W_t \neq \emptyset$, we set 
$$\epsilon(t)=1, \quad 
(\mbox{resp.~}\epsilon(t)=-1,)$$
if
(\ref{good2}) (resp.~(\ref{good1})) happens at $t$
for $v\in W_{t}$. 
By the condition (\ref{extra}), the value $\epsilon(t)$ does 
not depend on a choice of $v\in W_t$. 

In~\cite[Section~5]{Tcurve1}, 
we discussed the wall-crossing formula 
for invariants without Behrend function. 
If we assume Conjecture~\ref{conj:BG},
then the same argument is applied 
for invariants with Behrend function, 
without any major modification. 
(See~\cite[Section~8]{Tcurve1} in the arXiv version.)
Combined the main result of~\cite[Section~5]{Tcurve1}
with Conjecture~\ref{conj:BG}), 
we have the following: 
\begin{thm}\label{thm:summ}
\emph{\cite[Theorem~5.8, Corollary~5.11]{Tcurve1}}
Let $\vV \subset \Stab_{\Gamma_{\bullet}}(\dD_X)$ be
a connected subset satisfying Assumption~\ref{assum}.
Assuming Conjecture~\ref{conj:BG}, 
we have the following. 
\begin{itemize}
\item For $\sigma=(Z, \aA) \in \vV$ and
$v=(-n, -\beta, 1)\in \Gamma$,
(resp.~$v=(-n, -\beta, 0)\in \Gamma_0$,)
there 
is a counting invariant of $\sigma$-semistable 
objects of numerical type $v$,
$${}{\DT}_{n, \beta}(\sigma)\in \mathbb{Q}, 
\quad (\mbox{resp. }{}{N}_{n, \beta}\in \mathbb{Q},)$$
such that if 
$\mM^v(\sigma)$ is written as 
$[M/\mathbb{G}_m]$ for a $\mathbb{C}$-scheme $M$ with a trivial  
$\mathbb{G}_m$-action, we have (cf.~Remark~\ref{rmk:HT2},)
\begin{align}\label{vir}
{}{\DT}_{n, \beta}(\sigma)=\int_{[M]^{\rm{vir}}}1, \quad 
\left(resp.~{}{N}_{n, \beta}=\int_{[M]^{\rm{vir}}}1.\right)
\end{align}
\item 
Let ${}{\DT}(\sigma)$ and ${}{\DT}_{0}(\sigma)$
be the series, 
\begin{align}\label{DTse1}
{}{\DT}(\sigma)&=
\sum_{n, \beta}{}{\DT}_{n, \beta}(\sigma)x^n y^{\beta}, \\
\label{DTse2}
{}{\DT}_{0}(\sigma)&=
\sum_{(n, \beta)\in \Gamma_{0}}
{}{\DT}_{n, \beta}(\sigma)x^n y^{\beta}.
\end{align}
Then we have the following
equalities of the generating series, 
\begin{align}\label{cor:gen}
{}{\DT}(\tau)&=
{}{\DT}(\sigma)\cdot \prod_{\begin{subarray}{c}
-(n, \beta) \in W_t, \\
t \in (0, 1).
\end{subarray}}
\exp((-1)^{n-1}n{}{N}_{n, \beta}x^n y^{\beta})^{\epsilon(t)}, \\
\label{cor:gen2}
{}{\DT}_{0}(\tau)&=
{}{\DT}_{0}(\sigma)\cdot \prod_{\begin{subarray}{c}
-(n, \beta) \in W_t, \\
t \in (0, 1).
\end{subarray}}
\exp((-1)^{n-1}n{}{N}_{n, \beta}x^n y^{\beta})^{\epsilon(t)}.
\end{align}
In particular the quotient series
$${}{\DT}'(\sigma)\cneq \frac{{}{\DT}(\sigma)}{{}{\DT}_{0}(\sigma)},$$
is well-defined and does not depend on a general point 
$\sigma \in \vV$. 
\end{itemize}
 \end{thm}
 
  \begin{rmk}\label{rmk:HT2}
 Suppose that $\mM^{v}(\sigma)=[M/\mathbb{G}_m]$
 where $M$ is a scheme with a trivial $\mathbb{G}_m$-action. 
 Then
there is a perfect symmetric obstruction theory 
 on $M$ by~\cite{HT2}, and hence the associated 
 virtual cycle also exists. 
 \end{rmk}
  \begin{rmk}
  If $\mM^{v}(\sigma)$ is not written as $[M/\mathbb{G}_m]$
  as in Remark~\ref{rmk:HT2},
 then the invariant $\DT_{n, \beta}(\sigma)$, 
 (resp.~$N_{n, \beta}$,) 
 is defined by the integration of the 
 logarithm of the relevant moduli stacks in the Hall-algebra. 
 See~\cite[Definition~4.11]{Tcurve1}. 
 \end{rmk}
 
 \begin{rmk}\label{rmk:N}
 As in~\cite[Proposition-Definition~5.7]{Tcurve1}, the invariant 
 $N_{n, \beta}$ does not depend on $\sigma\in \vV$.
 However it may depend on $\vV$, so we may write it as 
 $N_{n, \beta}(\vV)$. Let $\vV_1, \vV_2 
 \subset \Stab_{\Gamma_{\bullet}}(\dD_X)$ be 
 connected subsets satisfying Assumption~\ref{assum}. 
 If $\vV_1 \cup \vV_2$ is connected, then the same proof 
 of~\cite[Proposition-Definition 5.7]{Tcurve1}
 show that 
 $$N_{n, \beta}(\vV_1)=N_{n, \beta}(\vV_2),$$
 i.e. the invariant $N_{n, \beta}$ does not depend on 
 $\sigma \in \vV_1 \cup \vV_2$.  
 \end{rmk}

\subsection{Completions of $\mathbb{C}[N_{\le 1}(X)]$}
In this subsection, we discuss certain 
completions of the group ring 
$\mathbb{C}[N_{\le 1}(X)]$, in which 
the generating series ${}{\DT}(\sigma)$
and ${}{\DT}_{0}(\sigma)$ are defined. 
For subsets $S_1, S_2 \subset N_{\le 1}(X)= \mathbb{Z}\oplus 
N_1(X)$,  
we set
$$S_1+S_2 \cneq \{s_1+s_2 : s_i \in S_i \}\subset
N_{\le 1}(X).$$
The sixth condition of Assumption~\ref{assum} 
is stated as follows. 
\begin{assum}\emph{\bf{\cite[Assumption~4.4]{Tcurve1}}}
\label{assum2}
In the situation of Assumption~\ref{assum}, 
the subsets $0\in T\subset S\subset N_{\le 1}(X)$
satisfy the following conditions. 
\begin{itemize}
\item We have 
\begin{align}\label{assum:pro}
T+T\subset T, \quad S+T \subset S.
\end{align}
\item For any $x\in N_{\le 1}(X)$, 
there are only finitely many ways to write 
$x=y+z$ for $y, z\in S$. 
\item 
Let $\psi \in \mathbb{R}$ be as in (\ref{phase1})
for $\sigma \in \vV$.  
For $I=(\psi-\varepsilon, \psi+\varepsilon)$
with $0<\varepsilon \ll 1$, we have 
\begin{align}
\label{S1}
&\{(n, \beta)\in N_{\le 1}(X) : 
(-n, -\beta, 1)\in C_{\sigma}(I)\}
\subset S, \\
\label{T1}
&\{(n, \beta)\in \Gamma_{0} : 
(-n, -\beta, r)\in C_{\sigma}(I), \ 
r=0 \mbox{ or } 1\}
\subset T.
\end{align}
Here $C_{\sigma}(I)\subset \Gamma$ is defined in (\ref{Csigma}). 
\item There is a family of sets
$\{S_{\lambda}\}_{\lambda \in \Lambda}$
with $S_{\lambda}\subset S$ such that 
$S\setminus S_{\lambda}$ is a finite set and 
\begin{align*}
S_{\lambda}+T\subset S_{\lambda}, \quad
 S=\bigcup_{\lambda\in \Lambda}(S\setminus S_{\lambda}).
\end{align*}
\end{itemize}
 \end{assum}
The existence of such $S, T$ are required to 
give completions of the group ring 
$\mathbb{C}[N_{\le 1}(X)]$. 
We have the following $\mathbb{C}$-vector space, 
$$\kakkoS\cneq \left\{ f=\sum_{(n, \beta)\in S}a_{n, \beta}x^n y^{\beta}
: a_{n, \beta}\in \mathbb{C} \right\}.$$
The vector spaces $\kakkoT$, 
$\kakkoSl$ are similarly defined. 
The product on 
$\mathbb{C}[N_{\le 1}(X)]$ generalizes naturally to 
products on $\kakkoT$, 
and $\kakkoS$, $\kakkoSl$ are $\kakkoT$-modules
with $\kakkoSl \subset \kakkoS$. 
There is a topology on $\kakkoS$, induced by 
the isomorphism, 
\begin{align}\notag
\kakkoS
\cong
\lim_{\begin{subarray}{c}\longleftarrow \\
\lambda \in \Lambda
\end{subarray}}
\kakkoS/
\kakkoSl,
\end{align}
and the Euclid topology on the finite dimensional 
vector spaces $\kakkoS/\kakkoSl$. 
(cf.~\cite[Section~4]{Tcurve1}.)
For $\sigma \in \vV$ satisfying Assumption~\ref{assum},
the third condition of Assumption~\ref{assum2} yields,   
$${}{\DT}(\sigma)\in \kakkoS, \quad
{}{\DT}_0(\sigma) \in \kakkoT,$$
where ${}{\DT}(\sigma)$, ${}{\DT}_0(\sigma)$
are given in (\ref{DTse1}), (\ref{DTse2}). 
\subsection{Checking assumptions}
Let $f\colon X\to Y$ be a
flopping contraction. 
Here we state
 that the connected subsets 
$\uU_{X, B}$ and $\pU_{X/Y}$
(cf.~Definition~\ref{def:UX},
Definition~\ref{def:pU},)
satisfy Assumption~\ref{assum}. 
For $\beta \in \NE(X)$, we set 
$m(\beta)$ as follows, 
\begin{align}\label{mbeta}
m(\beta)=\inf\{\chi(\oO_C) : \dim C =1 \mbox{ with }
[C] \le \beta\}.
\end{align}
It is well-known that $m(\beta)>-\infty$. 
(cf.~\cite[Lemma~3.10]{Tolim}.)
We set $S_X$ and $T_X$ as 
\begin{align}\label{good}
 S_X&\cneq \{ (n, \beta) \in N_{\le 1}(X) : 
 \beta \ge 0, \ n\ge m(\beta)\}, \\
 \label{verygood}
 T_X& \cneq \{(n, \beta) \in \Gamma_{0} :
 \beta\ge 0, \ n\ge 0\}.
 \end{align}
\begin{prop}\label{prop:check}
For $B\in \pV(X/Y)$, 
(cf.~(\ref{pV}),)
the subset $\uU_{X, B} \subset \Stab_{\Gamma_{\bullet}}(\dD_X)$
satisfies Assumption~\ref{assum} with 
$S=S_X$ and $T=T_X$. 
\end{prop}
\begin{proof}
The proof will be given in Section~\ref{sec:tech}.
\end{proof}
For $p=0, -1$, let $\pE=\oO_X \oplus \pE'$ be the vector bundle on $X$ 
constructed in 
the proof of Theorem~\ref{thm:nc}. 
We denote by $r(p)$ the rank of $\pE'$. 
For $v=(n, \beta)\in N_{\le 1}(X)$, we 
set $\pchi(v)$ as 
\begin{align*}
\pchi(v)&=\int_{X} v\cdot \ch \pE^{'\vee}, \\
&=r(p)n +(-1)^{p}c_1(\lL_X)\cdot \beta, 
\end{align*}
where $\lL_X$ is a globally generated ample 
line bundle which defines $\pE$. 
(See Theorem~\ref{thm:nc}.)

For an effective class $\beta \in N_1(Y)$, 
we set $\ppm(\beta)$ as follows, 
$$\ppm(\beta)=\inf\left\{ \chi(F) : 
\begin{array}{l} F\in \Coh_{\le 1}(Y), \ [F]\le \beta, 
\mbox{ there is a }\\
\mbox{surjection of sheaves }f_{\ast}\pE^{'\vee}\twoheadrightarrow 
F. \end{array}\right\},$$
as an analogue of (\ref{mbeta}). 
The same proof of~\cite[Lemma~3.10]{Tcurve1}
shows that $\ppm(\beta)>-\infty$. 
We set $\pS_{X/Y}$ and $\pTT_{X/Y}$ to be 
\begin{align*}
\pS_{X/Y}&= \left\{ v=(n, \beta)\in N_{\le 1}(X) : 
\begin{array}{l} f_{\ast}\beta \ge 0, \ n\ge m(f_{\ast}\beta), \\
\pchi(v) \ge \ppm (r(p)f_{\ast}\beta)
\end{array} \right\}, \\
\pTT_{X/Y}&= \left\{ v=(n, \beta)\in N_{\le 1}(X) : 
\begin{array}{l} f_{\ast}\beta \ge 0, \ n\ge 0, \\
\pchi(v) \ge 0
\end{array} \right\},
\end{align*}
\begin{prop}\label{prop:check2}
(i) 
For $B\in \pV(X/Y)$, 
the subset $\uU_{X, B} \subset \Stab_{\Gamma_{\bullet}}(\dD_X)$
satisfies Assumption~\ref{assum} with 
$S=\pS_{X/Y}$ and $T=\pTT_{X/Y}$. 

(ii) The subset $\pU_{X/Y}\subset \Stab_{\Gamma_{\bullet}}(\dD_X)$
satisfies Assumption~\ref{assum} with $S=\pS_{X/Y}$ and 
$T=\pTT_{X/Y}$. 
\end{prop}
\begin{proof}
The proof will be given in Section~\ref{sec:tech}. 
\end{proof}

\section{Proof of the main theorem}\label{sec:proof}
In this section, using all the notions and results in 
the previous subsections, we give a proof of 
Theorem~\ref{conj:main}. 
This section is the heart of this paper, 
which includes argument in Subsection~\ref{subsec:out}. 
Again $f\colon X\to Y$ is a flopping contraction 
from a smooth projective Calabi-Yau 3-fold $X$, and 
$\phi\colon X^{+} \dashrightarrow X$ its flop. 
\subsection{Counting invariants of rank zero objects}
Let $\uU \subset \Stab_{\Gamma_{\bullet}}(\dD_X)$ be the 
connected subset given by (\ref{connU}), and take 
$$v=(-n, -\beta, 0) \in \Gamma_{0}.$$
By Proposition~\ref{VU} (ii), 
Proposition~\ref{prop:check2}, Theorem~\ref{thm:summ} and 
Remark~\ref{rmk:N}, 
there is a counting invariant of 
$\sigma$-semistable objects of numerical type $v$, 
\begin{align}\label{Ninv}
N_{n, \beta} \in \mathbb{Q}, 
\end{align}
which does not depend on $\sigma \in \uU$. 
\begin{lem}\label{eq:N}
We have the following equality. 
$$N_{n, \beta}=N_{-n, -\beta}=N_{-n, \beta}.$$
\end{lem}
\begin{proof}
The first equality is just the definition of 
$N_{-n, -\beta}$. (cf.~\cite[Proposition-Definition~5.7]{Tcurve1}.) 
In order to show $N_{n, \beta}=N_{-n, \beta}$, 
take a data $\xi=(1, -i\omega, -i\omega', z)$ as in (\ref{data1}) 
with $B=0$. 
Then it is easy to see that 
an object $E\in \aA_{X}$ is $Z_{\xi}$-semistable 
if and only if $E[1]\in \Coh_{\le 1}(X)$
 is a $\omega$-Gieseker semistable sheaf. 
Therefore the dualizing functor 
$$\mM^{(-n, -\beta, 0)}(\sigma_{\xi}) \ni E \mapsto 
\dR \hH om(E, \oO_X) \in \mM^{(-n, \beta, 0)}(\sigma_{\xi}), $$
is an isomorphism, hence $N_{n, \beta}=N_{-n, \beta}$ holds. 
(See~\cite[Lemma~4.3]{Tolim2}.)
\end{proof}
\begin{rmk}
By the proof of Lemma~\ref{eq:N}, 
the invariant (\ref{Ninv}) coincides with Joyce-Song's
generalized DT-invariant~\cite{JS}, which counts $\omega$-Gieseker
semistable sheaves $E \in \Coh_{\le 1}(X)$, 
satisfying $[E]=\beta, \chi(E)=n$. 
Since Conjecture~\ref{conj:BG} holds 
for $E \in \Coh_{\le 1}(X)$, 
(cf.~\cite[Theorem~5.5]{JS},)
 or $E \in \pPPer_{0}(X/Y)$, 
(since $\pPPer_0(X/Y)$ is equivalent to 
the category of reprensentations of a quiver 
with a potential,)
we can easily see that 
(\ref{Ninv}) does not depend on $\sigma$
without assuming Conjecture~\ref{conj:BG}. 
\end{rmk}
\subsection{Semistable objects of rank one}
Let $\sigma=(Z, \pP) \in \uU \subset \Stab_{\Gamma_{\bullet}}(\dD_X)$
be as in the previous subsection, and take 
$$v=(-n, -\beta, 1) \in \Gamma.$$
In this subsection, we investigate the
moduli stack of $\sigma$-semistable objects, 
$$\mM^v(\sigma)\subset \oO bj(\aA), $$
which is algebraic by Proposition~\ref{prop:check} and 
Proposition~\ref{prop:check2}. 
We first note the following lemma, 
which follows from (\ref{writeA}) and (\ref{pB})
immediately. 
\begin{lem}\label{lem:filtAB}
For $E\in \aA_X$ (resp.~$E\in \pB_{X/Y})$
satisfying $\rk(E)=1$, there is a filtration 
in $\aA_X$, (resp.~$\pB_{X/Y})$,)
\begin{align}\label{filtA}
0=E_{-1} \subset E_0 \subset E_1 \subset E_2=E, 
\end{align}
such that each subquotient 
$F_i=E_i/E_{i-1}$ satisfies 
$$F_0, F_2 \in \Coh_{\le 1}(X)[-1],
 (\mbox{resp.}~\ppPPer_{\le 1}(X/Y)[-1],)
\quad 
F_1=\oO_X.$$
In particular if $E\in \aA_X$, there is an exact 
sequence in $\aA_X$, 
\begin{align}\label{exA}
0 \lr I_C \lr E \lr F[-1] \lr 0, 
\end{align}
where $C\subset X$ is a 1-dimensional subscheme 
with the defining ideal $I_C\subset \oO_X$, 
and $F\in \Coh_{\le 1}(X)$. 
\end{lem}
Let us fix 
\begin{align}\label{fixdata}
B\in \pV(X/Y), \quad \omega' \in A(Y), \quad
z\in \mathfrak{H} \mbox{ with }\arg z\in (\pi/2, \pi),
\end{align}
and deform $\omega=tH$ with $t\in \mathbb{R}_{>0}$. 
Here $H\in A(X/Y)$ is an ample generator given in (\ref{N^1}).
We obtain a 1-parameter family of 
weak stability conditions 
\begin{align}\label{Ufam}
\sigma_{\xi(t)}=(Z_{\xi(t)}, \aA_{X}) \in \uU_{X, B}, 
\end{align}
(cf.~Definition~\ref{def:UX},)
where $\xi(t)$ is 
\begin{align}\label{xit}
\xi(t)=(1, -(B+itH), -i\omega', z),
\end{align}
which is a family of data (\ref{data1}). 
\begin{prop}\label{prop:semi1}
For a fixed $v=(-n, -\beta, 1) \in \Gamma$
and the data (\ref{fixdata}),  
there is $t_0 \in \mathbb{R}$ such 
that for $t>t_0$, we have
\begin{align*}
\mM^{v}(\sigma_{\xi(t)})=[P_n(X, \beta)/\mathbb{G}_m],
\end{align*}
where $\mathbb{G}_m$ acts on $P_n(X, \beta)$ 
trivially. 
In particular,
the following holds in $\mathbb{C}\db[S_{X}\db]$,
$$\lim_{t\to \infty}\DT(\sigma_{\xi(t)})=\PT(X).$$
\end{prop}
\begin{proof}
Take a $\sigma_{\xi(t)}$-semistable object $E\in \aA_X$
with $\cl(E)=(-n, -\beta, 1)$. 
Let 
$$0 \to I_C \to E \to F[-1] \to 0,$$
be 
an exact sequence as in (\ref{exA}). 
We have
\begin{align}\label{ch_3}
\ch_3(F) \le n-m(\beta),
\end{align}
where $m(\beta)$ is defined by (\ref{mbeta}).
Suppose that the support of $F$ is 1-dimensional. 
The $\sigma_{\xi(t)}$-semistability of $E$ yields,
$$\arg Z_{\xi(t)}(F[-1])\ge \arg Z_{\xi(t)}(E)=\arg z>\pi/2.$$
Hence we have $f_{\ast}\ch_2(F)=0$ and 
\begin{align}\label{-ch}
\frac{-\ch_3(F)+B\ch_2(F)}{tH\cdot\ch_2(F)} 
\le c<0, 
\end{align}
for $c=\Ree z/\Imm z$. The inequalities 
(\ref{ch_3}) and (\ref{-ch}) 
imply
\begin{align}\label{tC}
t\le \frac{-n+m(\beta)}{cH\cdot\ch_2(F)}+\frac{b}{c},
\end{align}
where $B=bH$ for $b\in \mathbb{R}$. 
Since $0<H\cdot \ch_2(F)\le H\cdot \beta$, 
there is $t_0>0$ 
(depending only on $v$, $B$ and $z$,)
such that (\ref{tC}) implies 
$t\le t_0$. Therefore if we take $t>t_0$, the sheaf
 $F$ must be 0-dimensional. 
 Also we have $\Hom(\oO_x[-1], E)=0$ for any 
 closed point $x\in X$, since 
 $\oO_x[-1]$ is $\sigma_{\xi(t)}$-stable with 
 $$\pi=\arg Z_{\xi(t)}(\oO_x[-1])>\arg Z_{\xi(t)}(E)
 =\arg z.$$
 Then we apply~\cite[Lemma~3.11]{Tcurve1} 
 and conclude that $E$ is a stable pair~(\ref{twoterm}). 
 
 Conversely take a stable pair 
$E=(\oO_X \to F) \in \aA_X$
with $[F]=\beta$, $\chi(F)=n$, 
and an exact sequence in $\aA_X$, 
$$0 \lr A \lr E \lr B \lr 0.$$
Since there is a surjection of sheaves 
$\hH^1(E) \twoheadrightarrow \hH^1(B)$
and $\hH^1(E)$ is 0-dimensional, 
the sheaf $\hH^1(B)$ is also 0-dimensional. 
If $\rk(B)=0$, then $B=Q[-1]$ for a
0-dimensional sheaf $Q$. Hence the inequality 
$$\arg z=\arg Z_{\xi(t)}(E)<\arg Z_{\xi(t)}(B)=\pi$$
is satisfied. If $\rk(B)\neq 0$, we have 
$\rk(B)=1$ and $\rk(A)=0$. 
By the exact sequence (\ref{exA})
applied for $B\in \aA_X$,  
we see that
 $\ch_3(B)\le -m(\beta)$.
Hence
 we have $A=G[-1]$ for $G\in \Coh_{\le 1}(X)$
  with 
 $$\ch_3(G) \le n-m(\beta).$$
Therefore we have 
 \begin{align}\label{argin}
 \frac{-\ch_3(G)+B\ch_2(G)}{tH\cdot \ch_2(G)} 
 \ge \frac{-n+m(\beta)}{tH\cdot\ch_2(F)}+\frac{b}{t},
 \end{align}
 where $B=bH$ with $b\in \mathbb{R}$. 
 Hence
 there is $t_0>0$ such that 
 the RHS of 
 $(\ref{argin})$ is bigger than $c=\Ree z/\Imm z<0$
 for $t>t_0$, i.e. 
 $$\arg Z_{\xi(t)}(G[-1])<\arg Z_{\xi(t)}(E)=\arg z,$$
  for $t>t_0$.
\end{proof}

Next let us take $\xi=(-z_0, z_0 B, -i\omega', z_1)$
as in (\ref{element2}), and the associated
weak stability condition $\sigma_{\xi}=(Z_{\xi}, \pB_{X/Y})
 \in \pU_{X/Y}$.
(cf.~Definition~\ref{def:pU}.) 
We have the following proposition. 
\begin{prop}\label{prop:semi2}
(i) Suppose that $\arg z_0> \arg z_1$. 
Then for 
$v=(-n, -\beta, 1) \in \Gamma$ with 
$(n, \beta)\in \Gamma_0$, we have 
$$\mM^{v}(\sigma_{\xi})=\left\{\begin{array}{cc}
[\Spec\mathbb{C}/\mathbb{G}_m], & n=\beta=0, \\
\emptyset, & \mbox{\rm{otherwise}}.
\end{array}\right.$$
In particular, we have 
$$\DT_{0}(\sigma_{\xi})=1.$$
(ii) Suppose that $\arg z_0<\arg z_1$. 
Then for $v=(-n, -\beta, 1) \in \Gamma$, we have 
$$\mM^{v}(\sigma_{\xi})
=[I_n(\pAA_{Y}, \beta)/\mathbb{G}_m],$$
where $\mathbb{G}_m$ acts on $I_n(\pAA_Y, \beta)$ trivially. 
In particular, we have 
$$\DT(\sigma_{\xi})=\DT(\pAA_Y).$$
\end{prop}
\begin{proof}
(i) 
Take a $\sigma_{\xi}$-semistable object $E\in \pB_{X/Y}$
with $\cl(E)=v$, and 
a filtration
$$0=E_{-1}\subset E_0 \subset E_1 \subset E_2=E, $$
in $\pB_{X/Y}$ as in (\ref{filtA}). 
For each subquotient $F_i$, the condition 
$(n, \beta)\in \Gamma_0$ implies 
$$F_0, F_2 \in \ppPPer_{0}(X/Y)[-1], \quad F_1=\oO_X.$$
(cf.~Definition~\ref{def:pp}.)
Suppose that $F_0\neq 0$. Then we have
$$\arg Z_{\xi}(F_0)=\arg z_0> \arg z_1=\arg Z_{\xi}(E),$$
which contradicts to 
the $\sigma_{\xi}$-semistability of 
$E$. Hence $F_0=0$ and we have the exact sequence 
in $\pB_{X/Y}$, 
\begin{align}\label{exsp}
0 \lr \oO_X \lr E \lr F_2 \lr 0.
\end{align}
Since 
\begin{align*}
\Hom(F_2, \oO_X[1]) &\cong \hH^2 \dR \Gamma (X, F_2)^{\vee}, \\
&=0, 
\end{align*}
by the Serre duality and the definition of $\ppPPer_{0}(X/Y)$, 
the sequence (\ref{exsp}) splits. 
Hence $E$ is $\sigma_{\xi}$-semistable 
if and only if $E\cong\oO_X$. 

(ii)
Let 
$$\pH^{i}\colon 
D^b(\Coh(X)) \lr \ppPPer(X/Y)$$
be the $i$-th cohomology 
functor with respect to the t-structure on $D^b(\Coh(X))$
with heart $\ppPPer(X/Y)$. Take 
a $\sigma_{\xi}$-semistable object 
$E\in \pB_{X/Y}$ with $\cl(E)=v$, and 
suppose that $\pH^1(E)$ is non-zero. 
We have 
the surjection in $\pB_{X/Y}$, 
$$E \twoheadrightarrow \pH^1(E)[-1], $$
and the inequality,
$$\arg Z_{\xi}(E)=\arg z_1>\arg z_0=\arg Z_{\xi}(\pH^1(E)[-1]),$$
by our choice of $\xi$. 
This contradicts to the $\sigma_{\xi}$-semistablility
of $E$, 
hence we have $\pH^1(E)=0$. 
Combined with Lemma~\ref{lem:filtAB},
the object $E$ fits into the exact sequence 
in $\ppPPer(X/Y)$, 
\begin{align}\label{ExB}
0 \to E \to \oO_X \to F \to 0,
\end{align}
with $F\in \ppPPer_{\le 1}(X/Y)$. 

On the other hand, take $E\in \ppPPer(X/Y)$ which 
fits into an exact sequence (\ref{ExB}) in $\ppPPer(X/Y)$. 
Note that we have $E\in \pB_{X/Y}$ with $\pH^1(E)=0$. 
In order to show that $E$ is $\sigma_{\xi}$-stable, 
let us take an exact sequence 
in $\pB_{X/Y}$, 
$$0 \lr E_1 \lr E \lr E_2 \lr 0,$$
such that $E_i \neq 0$ for $i=1, 2$. 
Suppose that $\rk(E_1)=1$ and
$\rk(E_2)=0$, hence 
$E_2\in \pPPer_{\le 1}(X/Y)[-1]$. 
The long exact sequence associated to 
$\pH^{\bullet}(\ast)$
together with $\pH^1(E)=0$
show that $E_2=0$. 
This is a contradiction, hence 
$\rk(E_1)=0$ and $\rk(E_2)=1$ holds. In this case, 
our choice of $\xi$ yields, 
$$\arg Z_{\xi}(E_1)=\arg z_0< 
\arg z_1=\arg Z_{\xi}(E),$$ 
which implies that $E$ is $\sigma_{\xi}$-stable. 
\end{proof}

\subsection{Local transformation formula of the
generating series}
In this subsection, we show the 
transformation formula of ${}{\DT}(X/Y)$ 
and ${}{\DT}_{0}(\pAA_Y)$. 
\begin{thm}\label{thm:form}
Assuming Conjecture~\ref{conj:BG}, 
we have the formula, 
\begin{align}
\label{form:DT}
{}{\DT}(X/Y) &= \prod_{\begin{subarray}{c}
n>0, \beta \ge 0, \\
f_{\ast}\beta=0
\end{subarray}}\exp((-1)^{n-1}n{N}_{n, \beta}
x^n y^{\beta}). 
\end{align}
In particular, we have 
\begin{align}
\label{form1}
{}{\DT}(X/Y)&=i\circ \phi_{\ast}{}{\DT}(X^{+}/Y).
\end{align}
Here the variable change is $\phi_{\ast}(n, \beta)=(n, \phi_{\ast}\beta)$
and $i(n, \beta)=(n, -\beta)$. 
\end{thm}
\begin{proof}
Let us take 
$$\xi(t)=(1, -(B+itH), -i\omega', z),$$
as in (\ref{fixdata}), (\ref{xit}) and 
a 1-parameter family of weak stability
conditions 
$\sigma_{\xi(t)}\in \uU_{X, B}$
as in (\ref{Ufam}). 
By Proposition~\ref{prop:semi1},
we have 
$$\lim_{t\to \infty}{}{\DT}_{0}(\sigma_{\xi(t)})=
{}{\PT}(X/Y), $$
in the topological ring $\mathbb{C}\db[T_{X}\db]$. 
On the other hand, let us consider the element, 
$$\xi(0)=(1, -B, -i\omega', z),$$
which gives  
data (\ref{data2}) with $z_0=-1$ in the notation of (\ref{data2}).  
Note that ${}{\DT}_{0}(\sigma_{\xi(t)})$ and 
${}{\DT}_{0}(\sigma_{\xi(0)})$ are
contained in $\mathbb{C}\db[\pTT_{X/Y}\db]$
by Proposition~\ref{prop:check2}. 
Since $\sigma_{\xi(0)}$ is a general point, we have 
\begin{align*}
\lim_{t\to 0}{}{\DT}_{0}(\sigma_{\xi(t)})
&={}{\DT}_{0}(\sigma_{\xi(0)}) \\
&=1,
\end{align*}
in $\mathbb{C}\db[\pTT_{X/Y}\db]$. 
Here the second equality is due to Proposition~\ref{prop:semi2} (i). 
Therefore applying 
(\ref{cor:gen2}) in 
Theorem~\ref{thm:summ}, we have 
\begin{align}\notag
{}{\PT}(X/Y)&=\prod_{\begin{subarray}{c}
-(n, \beta)\in W_{t}, \\
0<t<\infty
\end{subarray} }
\exp((-1)^{n-1}nN_{n, \beta}x^n y^{\beta}) \\
\notag
&=\prod_{\begin{subarray}{c}
n-B\beta>0, \beta >0, \\
f_{\ast}\beta=0
\end{subarray}}\exp((-1)^{n-1}n{}{N}_{n, \beta}x^n y^{\beta}) \\
\label{Bto0}
&=\prod_{\begin{subarray}{c}
n>0, \beta >0, \\
f_{\ast}\beta=0
\end{subarray}}\exp((-1)^{n-1}n{}{N}_{n, \beta}x^n y^{\beta}).
\end{align}
Here the second equality follows from, 
$$\bigcup_{0<t<\infty}W_{t}=\{ -(n, \beta)\in \Gamma_{0}:
n-B\beta>0, \beta>0\} $$
and the last equality follows from taking the 
limit $B\to 0$ in $\pV(X/Y)$. 
Then (\ref{form:DT}) follows from (\ref{Bto0}), 
Theorem~\ref{mainTcurve} 
and the following formula, (cf.~\cite[Subsection~6.3]{JS},)
\begin{align}\label{McMahon}
\prod_{n>0}\exp((-1)^{n-1}n{}{N}_{n, 0}x^n)=
M(-x)^{\chi(X)}.
\end{align}
The formula (\ref{form1}) follows from (\ref{form:DT}), 
Lemma~\ref{eq:N} and the equality
$\chi(X)=\chi(X^{+})$. 
\end{proof}
Next we give the formula for ${}{\DT}_{0}(\pAA_Y)$. 
\begin{thm}\label{thm:form2}
Assuming Conjecture~\ref{conj:BG}, 
we have the following equality, 
\begin{align}
\label{form:nc}
{}{\DT}_{0}(\pAA_Y) &= \prod_{n>0, f_{\ast}\beta=0}
\exp((-1)^{n-1}n{N}_{n, \beta}
x^n y^{\beta}).
\end{align}
In particular, we have 
\begin{align}
\label{form2}
{}{\DT}_{0}(\pAA_Y)&={}{\DT}(X/Y)\cdot 
\phi_{\ast}{}{\DT}'(X^{+}/Y).
\end{align}
\end{thm}
\begin{proof}
Take data (\ref{data2}), 
\begin{align*}
\xi&=(-z_0, z_0B, -i\omega', z_1), \quad \arg z_0<\arg z_1, \\
\xi'&=(-z_0', z_0'B, -i\omega', z_1),  \quad \arg z_0'>\arg z_1, 
\end{align*}
and the associated weak stability conditions
$\sigma_{\xi}, \sigma_{\xi'} \in \pU_{X/Y}$. 
We consider a family of weak stability conditions 
$$\sigma_{\xi(t)}=(Z_{\xi(t)}, \pB_{X/Y})\in
\pU_{X/Y},$$
which connects $\sigma_{\xi}$ and $\sigma_{\xi'}$, 
 where $\xi(t)$
is given by 
\begin{align*}
\xi(t)=t\xi +(1-t)\xi'.
\end{align*}
By Proposition~\ref{prop:semi2}, we have 
$${}{\DT}_{0}(\sigma_{\xi(0)})=1, \quad 
{}{\DT}_{0}(\sigma_{\xi(1)})={}{\DT}_{0}(\pAA_Y).$$
Take $t_0\in (0, 1)$ which satisfies 
$$t_0z_0+(1-t_0)z_0' \in \mathbb{R}_{>0}z_1.$$
We have 
\begin{align*}
W_{t_0} &= 
\{v\in N_{\le 1}(X) : Z_{\xi(t_0)}(v) \in \mathbb{R}_{>0}z_1\}, \\
&=\{-(n, \beta)\in N_{\le 1}(X): n-B\beta>0\}.
\end{align*}
Hence applying
(\ref{cor:gen2}) in Theorem~\ref{thm:summ}, we obtain 
\begin{align*}
{}{\DT}_{0}(\pAA_Y)&=
\prod_{n-B\beta>0, f_{\ast}\beta=0}\exp((-1)^{n-1}n{}{N}_{n, \beta}x^n y^{\beta}), \\
&=\prod_{n>0, f_{\ast}\beta=0}\exp((-1)^{n-1}n{}{N}_{n, \beta}x^n y^{\beta}).
\end{align*}
Here the second equality follows from taking 
the limit $B\to 0$ in $\pV(X/Y)$. 
Hence (\ref{form:nc}) holds. 
The formula (\ref{form2}) follows from (\ref{McMahon}), 
(\ref{form:nc}) 
and Theorem~\ref{thm:form}. 
\end{proof}

\subsection{Global transformation formula}
Finally we show the global 
transformation formula of our generating functions. 
\begin{thm}\label{thm:global}
Assuming Conjecture~\ref{conj:BG},
we have the following formula, 
\begin{align}\label{form:global}
\frac{{}{\DT}(X)}{{}{\DT}(X/Y)}
=\frac{{}{\DT}(\pAA_Y)}{{}{\DT}_{0}(\pAA_Y)}
=\phi_{\ast}\frac{{}{\DT}(X^{+})}{{}{\DT}(X^{+}/Y)}.
\end{align}
\end{thm}
\begin{proof}
Let us take 
$$\xi(t)=(1, -(B+itH), -i\omega', z), \ t\in \mathbb{R}_{>0}$$
as in (\ref{fixdata}), (\ref{xit}) and 
a 1-parameter family of weak stability
conditions 
$\sigma_{\xi(t)}\in \uU_{X, B}$
as in (\ref{Ufam}). 
By Proposition~\ref{prop:semi1} and Theorem~\ref{mainTcurve}, we have  
\begin{align}\label{eq1}
\lim_{t\to \infty}
\frac{{}{\DT}(\sigma_{\xi(t)})}{{}{\DT}_0(\sigma_{\xi(t)})}
=\frac{{}{\PT}(X)}{{}{\PT}(X/Y)}
=\frac{{}{\DT}(X)}{{}{\DT}(X/Y)},
\end{align}
in $\mathbb{C}\db[S_{X}\db]$. On the other hand, 
we have 
\begin{align}\label{eq2}
\lim_{t\to 0}
\frac{{}{\DT}(\sigma_{\xi(t)})}{{}{\DT}_0(\sigma_{\xi(t)})}
=\frac{{}{\DT}(\sigma_{\xi(0)})}{{}{\DT}_0(\sigma_{\xi(0)})},
\end{align}
since $\sigma_{\xi(0)}$ is a general point
and $\lim_{t\to 0}\sigma_{\xi(t)}=\sigma_{\xi(0)}$ 
by Proposition~\ref{VU}. 
Next let us take a data (\ref{data2}),
$$\xi=(-z_0, z_0B, -i\omega', z_1), \quad \arg z_0<\arg z_1. $$
By Theorem~\ref{thm:summ} and Proposition~\ref{prop:semi2}, we have 
\begin{align}\label{eq3}
\frac{{}{\DT}(\sigma_{\xi(0)})}{{}{\DT}_0(\sigma_{\xi(0)})}
=\frac{{}{\DT}(\sigma_{\xi})}{{}{\DT}_0(\sigma_{\xi})}
=\frac{{}{\DT}(\pAA_Y)}{{}{\DT}_{0}(\pAA_Y)}.
\end{align}
Finally suppose that $p=0$, i.e. $B\in \oV(X/Y)$. 
Note that we have $\phi_{\ast}^{-1}(B) \in \iV(X^{+}/Y)$.
For $t<0$, we set 
$$\xi(t)=(1, -\phi_{\ast}^{-1}(B+tiH), -i\omega', z), $$
which gives data (\ref{data1}) for $X^{+}$. We have the associated
1-parameter family of 
weak stability conditions $\sigma_{\xi(t)}^{+} \in 
\uU_{X^{+}, \phi_{\ast}^{-1}B}$, and we set 
$$\sigma_{\xi(t)}\cneq \Phi_{\ast}\sigma_{\xi(t)}^{+}
\in \Phi_{\ast}(\uU_{X^{+}, \phi_{\ast}^{-1}B}) \mbox{ for }
t<0.$$
By Proposition~\ref{VU}, the family 
$\sigma_{\xi(t)}$ is a continuous family for $t\in (-\infty, \infty)$. 
By Proposition~\ref{prop:semi1} and Theorem~\ref{mainTcurve}, 
we have 
\begin{align}\label{eq4}
\lim_{t\to -\infty}
\frac{{}{\DT}(\sigma_{\xi(t)})}{{}{\DT}_0(\sigma_{\xi(t)})}
=\phi_{\ast}\frac{{}{\PT}(X^{+})}{{}{\PT}(X^{+}/Y)}
=\phi_{\ast}\frac{{}{\DT}(X^{+})}{{}{\DT}(X^{+}/Y)}.
\end{align}
Then the formula (\ref{form:global}) 
follows from (\ref{eq1}), (\ref{eq2}), (\ref{eq3}), 
(\ref{eq4}) and Theorem~\ref{thm:summ}. 
\end{proof}

\section{Some technical lemmas}\label{sec:tech}
\subsection{Proof of Lemma~\ref{lem:noether}}
\begin{proof}
(i) The noetherian property of $\aA_X$ is proved in~\cite[Lemma~6.2]{Tcurve1}. 
Let us show that $\pB_{X/Y}$ is noetherian. 
For simplicity, we show the case of $p=0$. 
Take a chain of surjections in $\oB_{X/Y}$, 
\begin{align}\label{chain}
E_0 \twoheadrightarrow E_1 \twoheadrightarrow \cdots
\twoheadrightarrow E_j \twoheadrightarrow E_{j+1} \twoheadrightarrow 
\cdots.
\end{align}
The description (\ref{per:gen10}) shows that
$\oB_{X/Y}$ is concentrated on $[0, 1]$
with respect to the standard t-structure on $D^b(\Coh(X))$, 
and (\ref{chain}) induces a chain of surjections in $\Coh(X)$, 
$$\hH^1(E_0) \twoheadrightarrow 
 \hH^1(E_1) \twoheadrightarrow \cdots
\twoheadrightarrow \hH^1(E_j)
 \twoheadrightarrow \hH^1(E_{j+1}) \twoheadrightarrow 
\cdots.$$
Hence
we may assume that $\hH^1(E_i) \stackrel{\cong}{\to}\hH^1(E_{i+1})$
for all $i$. 
By (\ref{pB}) and the definition of $\oPPer(X/Y)$, 
we have the exact functor, 
\begin{align}\label{exact}
\dR f_{\ast} \colon \oB_{X/Y} \lr \langle \oO_Y, \Coh_{\le 1}(Y)[-1]
\rangle_{\ex}.
\end{align}
The proof that $\aA_X$ is noetherian
(cf.~\cite[Lemma~6.2]{Tcurve1}) is also applied for
the singular variety $Y$, hence 
the category $\langle \oO_Y, \Coh_{\le 1}(Y)[-1]
\rangle_{\ex}$ is also noetherian. 
Therefore
we may assume that $\dR f_{\ast}E_i \stackrel{\cong}{\to}
\dR f_{\ast}E_{i+1}$ for all $i$. 
Consider the exact sequence in $\oB_{X/Y}$, 
$$0 \lr F_{i} \lr E_0 \lr E_i \lr 0.$$
Noting (\ref{per:gen10}) and
 $\dR f_{\ast}F_i=0$, 
the object $F_i$ is written as
$F_i'[-1]$, where $F_i'$ is given by successive 
extensions of sheaves $\oO_{C_k}(-1)$
with $1\le k \le N$. Hence 
we obtain the exact sequence of sheaves, 
$$0 \lr \hH^0(E_0) \lr \hH^0(E_i) \lr F_i' \lr 0.$$
Since $\dim C_k \le 1$, we obtain 
the sequence of coherent sheaves, 
\begin{align}\label{dual}
\hH^0(E_0) \subset \hH^0(E_{1}) \subset \cdots \subset 
\hH^0(E_{j}) \subset \cdots \subset 
\hH^0(E_0)^{\vee \vee}.
\end{align}
Since $\Coh(X)$ is noetherian,
the above sequence terminates.  

(ii) First we show the termination of (\ref{stmo2}) 
in $\aA_{X}$. Take a sequence (\ref{stmo2})
in $\aA_X$, and an ample divisor $\omega$ on $X$. 
By (\ref{writeA}), 
we have $\ch_0(E)\ge 0$ and $-\ch_2(E)\cdot \omega \ge 0$
for any object $E\in \aA_X$. 
Therefore we may assume that $\ch_0(E_i)$ and $\ch_2(E_i)\cdot \omega$
are constant for all $i$. Then $E_j/E_{j+1}$ is 0-dimensional, 
hence it must be zero by the assumption. 

Next we show the termination of (\ref{stmo2}) 
in $\pB_{X/Y}$. For simplicity we show the case of $p=0$. 
By the same argument as in (i), 
we may assume that $G_j=E_0/E_{j+1}$ in $\oB_{X/Y}$
is written as $G_j'[-1]$, where $G_j'$ 
is given by successive extensions of
sheaves $\oO_{C_k}(-1)$
with $1\le k\le N$. 
We have the surjections of sheaves
$$\hH^1(E_0)\twoheadrightarrow  \cdots 
\twoheadrightarrow \cdots  \twoheadrightarrow G_{2}' \twoheadrightarrow 
G_{1}'.$$
 Since 
   $\hH^1(E_0) \in \Coh_{\le 1}(X)$, 
   the above sequence must terminate, and hence
    (\ref{stmo2}) also terminates.  
    \end{proof}
\subsection{Proof of Lemma~\ref{lem:pair1}}
\begin{step}
The pair $\sigma_{\xi}=(Z_{\xi}, \aA_{X})$
satisfies the Harder-Narasimhan property. 
\end{step}
\begin{proof}
It is enough to check (a) and (b) in 
Proposition~\ref{suHN}.
The condition (b) follows from Lemma~\ref{lem:noether} (i).  
In order to check (a), take a chain of monomorphisms in $\aA_X$, 
$$\cdots \subset E_{j+1} \subset E_j \subset \cdots \subset E_2 \subset E_1$$
with $\arg Z(E_{j+1})>\arg Z(E_j/E_{j+1})$ for all $j$. 
By Lemma~\ref{lem:noether} (ii), 
we have $E_j/E_{j+1} \in \Coh_{0}(X)$ for some $j$. 
Then we have 
$\arg Z(E_{j+1})>\arg Z(E_j/E_{j+1})=\pi$, which 
contradicts to (\ref{def:weak}). 
\end{proof}
\begin{step}
Let $\{\pP_{\xi}(\phi)\}_{\phi \in \mathbb{R}}$
  be the slicing corresponding to the pair 
 $\sigma_{\xi}=(Z_{\xi}, \aA_X)$
 via Proposition~\ref{prop:corr}. 
 Then $\{\pP_{\xi}(\phi)\}_{\phi \in \mathbb{R}}$ 
 is of locally finite. 
\end{step}
\begin{proof}
We set $\phi=\frac{1}{\pi}\arg z \in (1/2, 1)$ and 
take $0<\eta \ll 1$ satisfying $\phi\pm \eta \in (0, 1)$. 
By Lemma~\ref{lem:noether}, it is enough to check that 
$\pP_{\xi}((\theta-\eta, \theta+\eta))$ is of finite
length for any $\theta \in (1-\eta, 1+\eta)$. 
Let us consider the pair, 
\begin{align}\label{X/Y}
(Z_{0, \xi}, \Coh(X/Y)), 
\end{align}
where 
$$\Coh(X/Y)=\{E\in \Coh(X) : \dim \Supp f_{\ast}E=0\}.$$
Then the pair (\ref{X/Y}) determines a locally finite 
stability condition on $D^b(\Coh(X/Y))$ in the sense 
of Bridgeland~\cite{Brs1}. (cf.~\cite[Lemma~4.1]{Tst}.)
We write the corresponding slicing on $D^b(\Coh(X/Y))$
by $\{\qQ(\phi)\}_{\phi \in \mathbb{R}}$. 
By our choice of $\eta$, we have 
$$\pP_{\xi}((\theta-\eta, \theta+\eta))=\qQ((\theta-\eta, \theta+\eta)), $$
for any $\theta \in (1-\eta, 1+\eta)$. Therefore 
$\pP_{\xi}((\theta-\eta, \theta+\eta))$ is of finite length. 
\end{proof}
\begin{step}
The pair $\sigma_{\xi}=(Z_{\xi}, \pP_{\xi})$ satisfies the 
support property (\ref{support}).
\end{step}
\begin{proof}
Let $E\in \aA_X$ be a non-zero object 
 with $\cl(E)=(-n, -\beta, r)$. 
 We introduce an usual Euclid 
 norm on 
 $\mathbb{H}_2 \otimes_{\mathbb{Z}} \mathbb{R}=\mathbb{R}$. 
 We have 
 \begin{align*}
 \frac{\lVert E \rVert}{\lvert Z(E) \rvert}
 =\left\{ \begin{array}{cc}
  \lvert z \rvert,  & r>0, \\
  \frac{\lVert f_{\ast}\beta \rVert}{\omega' \cdot f_{\ast}\beta},
   & r=0, \ f_{\ast}\beta\neq 0, \\
  \frac{n^2 +\lVert \beta \rVert^2}{(n-B\beta)^2 +(\omega \beta)^2}, 
  & r=f_{\ast}\beta=0, n>0.
 \end{array}
 \right. 
 \end{align*}
 Since $\beta$ is effective or zero, the above 
 description easily implies the support property. 
\end{proof}

\subsection{Proof of Proposition~\ref{prop:check}}
The conditions of Assumption~\ref{assum}
are obviously satisfied except the fourth, fifth and sixth 
conditions. As for the fourth condition, 
this is proved in~\cite[Lemma~3.15]{Tcurve1}. 
It is enough to check the fifth and the sixth conditions. 
\begin{sstep}
Take $\sigma_{\xi} \in \uU_{X, B}$ and 
$v\in \Gamma$ with $\rk(v)=1$
or $v\in \Gamma_0$. Then the stack 
$$\mM^v(\sigma_{\xi})\subset \oO bj^{v}(\aA_X),$$
is an open substack of finite type over $\mathbb{C}$. 
\end{sstep}
\begin{proof}
As in the proof of~\cite[Lemma~3.15]{Tcurve1},
it is enough to check the boundedness of $\sigma_{\xi}$-semistable 
objects of numerical type $v$. 
This is proved along with the same argument of~\cite[Section~3]{Tolim}, 
and we leave the readers to check the detail. 
\end{proof}
\begin{sstep}
The sets $S_X$ and $T_X$ satisfy 
Assumption~\ref{assum2}. 
\end{sstep}
\begin{proof}
The first and the second
conditions of Assumption~\ref{assum2} are obviously satisfied. 
In order to prove the third condition, take an object
$E\in C_{\sigma_{\xi}}(I)$ with
$\sigma_{\xi}=(Z_{\xi}, \aA_{X}) \in \uU_{X, B}$ and 
 $\cl(E)=(-n, -\beta, 1)$. 
 We are going to check that $(n, \beta)\in S_X$. 
We have the exact sequence in $\aA_X$
\begin{align}\label{exAA}
0 \lr I_C \lr E \lr F[-1] \lr 0, 
\end{align}
as in (\ref{exA}).
Also we have the exact sequence of sheaves, 
$$0 \lr F_1 \lr F \lr F_2 \lr 0, $$
with $F_1 \in \pT$ and $F_2 \in \pF$, 
where $(\pT, \pF)$ is a torsion
pair on $\Coh_{\le 1}(X)$ given by (\ref{pTpF}). 
Assume that $F_2 \neq 0$. Then we have the 
surjection $E\twoheadrightarrow F_2[-1]$ in $\aA_X$. 
Then it is easy to see that 
$$\arg Z_{\xi}(E)>\pi/2> \arg Z_{\xi}(F_2[-1]), $$
which contradicts to $E\in C_{\sigma_{\xi}}(I)$. 
Therefore $F_2=0$, and $F\in \pT \subset \ppPPer_{\le 1}(X/Y)$
follows. 
Also if $\dR f_{\ast}F$ is not 0-dimensional, we have 
$$\arg Z_{\xi}(E)>\arg Z_{\xi}(F)=\pi/2,$$
which contradicts to the $\sigma_{\xi}$-semistability 
of $E$. Therefore $\dR f_{\ast}F$ is 0-dimensional, which 
means 
$F\in \ppPPer_{0}(X/Y)$. 
This implies that 
\begin{align}\label{ineqch}
\ch_3(F)=\length \dR f_{\ast}F \ge 0.
\end{align}
On the other hand, the definition of $m(\beta)$ implies 
$\ch_3(\oO_C)\ge m(\beta)$. 
Hence by (\ref{exAA}) and (\ref{ineqch}), the inequality 
$n\ge m(\beta)$ holds, i.e. $(n, \beta)\in S_X$.
  
 Also if $(n, \beta)\in \Gamma_0$, 
then the curve $C$ 
in the sequence (\ref{exAA}) satisfies $f_{\ast}[C]=0$, 
hence we have $H^1(\oO_C)=0$. This implies that
$\ch_3(\oO_C)=\chi(\oO_C)\ge 0$, hence 
$(n, \beta)\in T_X$ holds. If $\cl(E)=(-n, -\beta, 0)$, 
then the same argument shows 
that $E=F[-1]$ for $F\in \pT$
and $\Supp(F) \subset \Ex(f)$.
Hence $\ch_3(F)\ge 0$, and $(n, \beta)\in T_X$ follows. 

Finally we check the last condition of Assumption~\ref{assum2}. 
Let
$\Lambda$ be the set of pairs $(k, \beta')$ of 
$k\in \mathbb{Z}$ and an effective class $\beta' \in N_1(X)$. 
For $\lambda=(k, \beta')$, we set 
$$S_{\lambda}=\{ (n, \beta) \in S :
n\ge k \mbox{ if }\beta\le \beta'\}.
$$
Then $\{S_{\lambda}\}_{\lambda \in \Lambda}$ 
gives a desired family. 
\end{proof}

\subsection{Proof of Proposition~\ref{prop:check2}}
\begin{proof}
(i) By Proposition~\ref{prop:check}, it is enough to check 
the third and the fourth condition of Assumption~\ref{assum2}. 
For $\sigma_{\xi} \in \uU_{X, B}$, let us take 
an object $E\in C_{\sigma_{\xi}}(I)$ with 
$\cl(E)=(-n, -\beta, 1)$.
We check that $(n, \beta)\in \pS_{X/Y}$. 
We have an exact sequence in $\aA_X$, 
\begin{align}\label{E'EF}
0 \lr E' \lr E \lr F[-1] \lr 0, 
\end{align}
for $F\in \pF$ and $E' \in \pT'$
by the existence of the torsion pair (\ref{pT}). 
By the condition $E\in C_{\sigma}(I)$, 
we have $F[1] \in \ppPPer_{0}(X/Y)$, hence 
$v=(\ch_3(F), \ch_2(F))$ satisfies 
\begin{align}\label{v1}
\ch_3(F)\ge 0, \quad \pchi(v)\ge 0, 
\end{align}
by Lemma~\ref{lem:pchi} below. 
On the other hand, since $E' \in \pT' \subset 
\pB_{X/Y}$, we have the exact sequence in 
$\pB_{X/Y}$, 
\begin{align}\label{AEB}
0\lr A \lr E' \lr A'[-1] \lr 0, 
\end{align}
with $A\in \ppPPer_{\ge 2}(X/Y)$ and $A'\in \ppPPer_{\le 1}(X/Y)$. 
Taking the long exact sequence of cohomology 
with respect to the t-structure with heart $\aA_X$, 
we see that $A\in \aA_X$ and 
obtain the following exact sequence in $\aA_X$, 
\begin{align*}
0 \lr \hH^{-1}(B)[-1] \lr A \lr E' \lr \hH^{0}(B)[-1] \lr 0.
\end{align*}
By (\ref{E'EF}) and 
the condition $E\in C_{\sigma_{\xi}}(I)$,
we conclude that $\hH^i(A')$ are supported on 
fibers of $f$, hence $A'\in \ppPPer_{0}(X/Y)$ follows. 
Therefore $v'=(\ch_3(A'), \ch_2(A'))$ satisfies 
\begin{align}\label{v2}
\ch_3(A') \ge 0, \quad \pchi(v') \ge 0,
\end{align}
by Lemma~\ref{lem:pchi}. 
By Lemma~\ref{lem:filtAB}, the object $A\in \ppPPer(X/Y)$
 fits into the exact 
sequence in $\ppPPer(X/Y)$, 
$$0 \lr A \lr \oO_X \lr A'' \lr 0.$$
Applying $\dR f_{\ast}$, we obtain surjections in $\Coh(Y)$, 
$$\oO_Y \twoheadrightarrow \dR f_{\ast}A'', \quad 
f_{\ast}\pE^{'\vee} \twoheadrightarrow \dR f_{\ast}(A''\otimes 
\pE^{'\vee}).$$
Combining (\ref{v1}) and (\ref{v2}), $v=(n, \beta)$ satisfies
\begin{align*}
n&\ge \chi(\dR f_{\ast}A'') \ge m(f_{\ast}\beta), \\ 
\pchi(v)&\ge \chi(\dR f_{\ast}(A''\otimes 
\pE^{'\vee}))\ge \ppm(r(p)f_{\ast}\beta), 
\end{align*}
which implies $(n, \beta)\in \pS_{X/Y}$. 
A similar proof shows that if $(n, \beta)\in \Gamma_0$, 
(or $\cl(E)=(-n, -\beta, 0),$)
then $(n, \beta)\in T_{X/Y}$.

(ii)
As for the sixth condition of Assumption~\ref{assum2}, 
a similar (and easier) proof to (i) works, and we omit the detail. 
Here we check the forth and the fifth conditions of Assumption~\ref{assum}.
\begin{ssstep}
For $v\in \Gamma$ with $\rk(v)=1$ or $v\in \Gamma_{0}$,
the stack of objects 
$$\oO bj^{v}(\pB_{X/Y}) \subset \mM,$$
is an open substack of $\mM$. 
\end{ssstep} 
\begin{proof}
Note that the category $\pB_{X/Y}$ is equivalent to 
the category, 
$$\langle \pAA_{Y}', \Coh_{\le 1}(\pAA_{Y}) \rangle_{\ex}, $$
via the equivalence (\ref{nc}). Then we can apply the same 
argument of~\cite[Lemma~3.15]{Tcurve1}
for the non-commutative scheme $(Y, \pAA_{Y})$, and 
obtain the result. 
\end{proof}
\begin{ssstep}
Take $\sigma_{\xi}=(Z_{\xi}, \pB_{X/Y}) \in \pU_{X/Y}$
and $v\in \Gamma$ with $\rk(v)=1$ or $v\in \Gamma_{0}$. 
Then the substack 
$$\mM^{v}(\sigma_{\xi}) \subset \oO bj^{v}(\pB_{X/Y}),$$
is an open substack and it is of finite type over $\mathbb{C}$. 
\end{ssstep}
\begin{proof}
As in the proof of~\cite[Lemma~3.15]{Tcurve1}, 
it is enough to show the boundedness of $\sigma_{\xi}$-semistable 
objects of numerical type $v$. This follows by the same argument as 
in~\cite[Section~3]{Tolim}, applied for the non-commutative 
scheme $(Y, \pAA_Y)$. We leave the readers to check the detail. 
\end{proof}
\end{proof}
We have used the following lemma. 
\begin{lem}\label{lem:pchi}
For $F\in \ppPPer_{0}(X/Y)$, set
$v=(\ch_3(F), \ch_2(F))\in N_{\le 1}(X)$.
Then we have 
$$\ch_3(F)\ge 0, \quad \pchi(v) \ge 0.$$
\end{lem}
\begin{proof}
For $F\in \ppPPer_{0}(X/Y)$, we have 
$$\dR \Hom(\oO_X \oplus \pE', F) \in \Coh_{0}(Y), $$
by the equivalence (\ref{nc}). Therefore by Riemann-Roch theorem, we have
\begin{align*}
\ch_3(F)&=\length \dR \Hom(\oO_X, F)\ge 0, \\
\pchi(v)&=\length \dR \Hom(\pE', F)\ge 0.
\end{align*}
\end{proof}

\section{Appendix}
\subsection{The formula for the Euler characteristic version}
\label{subsec:Euler}
Applying the result of~\cite[Section~4]{Tcurve1}
and the method in this paper, we can also show the 
Euler characteristic version of our main result. 
\begin{defi}\label{def:Euler}
\emph{We define $\widehat{\DT}(X)$, $\widehat{\DT}(X/Y)$, 
$\widehat{\DT}(\pAA_Y)$, $\widehat{\DT}_{0}(\pAA_Y)$ as follows. }
\begin{align*}
\widehat{\DT}(X)&=\sum_{n, \beta}\chi(I_n(X, \beta))x^n y^{\beta}, \\
\widehat{\DT}(X/Y)&=\sum_{n, f_{\ast}\beta=0}
\chi(I_n(X, \beta))x^n y^{\beta}, \\
\widehat{\DT}(\pAA_{Y})&=\sum_{n, \beta}\chi(I_n(\pAA_Y, \beta))
x^n y^{\beta}, \\
\widehat{\DT}_{0}(\pAA_{Y})&=\sum_{n, f_{\ast}\beta=0}
\chi(I_n(\pAA_Y, \beta))x^n y^{\beta}.
\end{align*}
\end{defi}
The following theorem can be proved along with the same 
proof of Theorem~\ref{thm:form}, Theorem~\ref{thm:form2} and 
Theorem~\ref{thm:global}, using~\cite[Theorem~3.13]{Tcurve1}
instead of Theorem~\ref{thm:summ}. 
\begin{thm}\label{thm:Euler}
We have the following formula, 
\begin{align*}
\widehat{\DT}(X/Y)&=i\circ \phi_{\ast}\widehat{\DT}(X^{+}/Y), \\
\widehat{\DT}_0(A_Y)&=\widehat{\DT}(X/Y) \cdot 
\phi_{\ast}\widehat{\DT}'(X^{+}/Y), \\
\frac{\widehat{\DT}(X)}{\widehat{\DT}(X/Y)}&=
\frac{\widehat{\DT}(\pAA_Y)}{\widehat{\DT}_0(\pAA_Y)}=\phi_{\ast}
\frac{\widehat{\DT}(X^{+})}{\widehat{\DT}(X^{+}/Y)}.
\end{align*}
\end{thm}
Since we do not assume Conjecture~\ref{conj:BG}
in~\cite[Theorem~3.13]{Tcurve1}, 
we do not rely on 
Conjecture~\ref{conj:BG}
in Theorem~\ref{thm:Euler}. 
Hence the above
Euler characteristic version
is mathematically rigorous. 
\subsection{Generalization of global ncDT-invariants}
The non-commutative Donaldson-Thomas invariant
can be defined in a slightly generalized context. 
Let $f\colon X\to Y$ be a projective birational 
morphism from a smooth projective Calabi-Yau 3-fold $X$, 
satisfying $\dim f^{-1}(y) \le 1$ for any 
closed point $y\in Y$. Here we do not assume that 
$f$ is isomorphic in codimension one, so there 
may be a divisor $E\subset X$ which contracts to a 
curve on $Y$. The result of Van den Bergh~\cite{MVB}
can be applied in this situation, i.e. 
there
are vector bundles $\pE$ 
on $X$ for $p=0, -1$, which admit derived equivalences, 
\begin{align}
\pPhi=\dR f_{\ast} \dR \hH om(\pE, \ast) \colon 
D^b(\Coh(X)) \stackrel{\cong}{\lr} D^b(\Coh(\pAA_{Y})).
\end{align}
The abelian subcategories $\ppPPer(X/Y)\subset D^b(\Coh(X))$
is also similarly defined, and $\pPhi$ restrict
to equivalences between $\ppPPer(X/Y)$ and $\Coh(\pAA_Y)$. 
As in subsection~\ref{subsec:Non}, we can construct the moduli space of
perverse ideal sheaves 
$I_n(\pAA_Y, \beta)$, and the counting 
invariant, 
$$\pAA_{n, \beta}=\int_{[I_n(\pAA_Y, \beta)]^{\rm{vir}}}1 \in \mathbb{Z}.$$
The generating series $\DT(\pAA_Y)$, $\DT_{0}(\pAA_Y)$, 
$\widehat{\DT}(\pAA_{Y})$ and $\widehat{\DT}_{0}(\pAA_Y)$
are similarly defined as in Definition~\ref{def:glncDT}, 
Definition~\ref{def:Euler}. 
The following theorem can be proved along with
the same proof of Theorem~\ref{thm:global} and Theorem~\ref{thm:Euler}. 
\begin{thm}
Assuming Conjecture~\ref{conj:BG}, 
we have the following formula. 
$$\frac{{\DT}(X)}{{\DT}(X/Y)}=
\frac{{\DT}(\pAA_Y)}{{\DT}_0(\pAA_Y)}, \quad 
\frac{\widehat{\DT}(X)}{\widehat{\DT}(X/Y)}=
\frac{\widehat{\DT}(\pAA_Y)}{\widehat{\DT}_0(\pAA_Y)}.$$
\end{thm}

Institute for the Physics and 
Mathematics of the Universe, University of Tokyo

\textit{E-mail address}:toda-914@pj9.so-net.ne.jp


\begin{thebibliography}{10}

\bibitem{Beh}
K.~Behrend.
\newblock Donaldson-{T}homas invariants via microlocal geometry.
\newblock {\em Ann.~of Math}, Vol. 170, pp. 1307--1338, 2009.

\bibitem{BBr}
K.~Behrend and B.~Fantechi.
\newblock Symmetric obstruction theories and {H}ilbert schemes of points on
  threefolds.
\newblock {\em Algebra Number Theory}, Vol.~2, pp. 313--345, 2008.

\bibitem{BG}
K.~Behrend and E.~Getzler.
\newblock Chern-{S}imons functional.
\newblock {\em in preparation}.

\bibitem{Br1}
T.~Bridgeland.
\newblock Flops and derived categories.
\newblock {\em Invent. Math}, Vol. 147, pp. 613--632, 2002.

\bibitem{Brs1}
T.~Bridgeland.
\newblock Stability conditions on triangulated categories.
\newblock {\em Ann.~of Math}, Vol. 166, pp. 317--345, 2007.

\bibitem{Ch}
J-C. Chen.
\newblock Flops and equivalences of derived categories for three-folds with
  only {G}orenstein singularities.
\newblock {\em J.~Differential.~Geom}, Vol.~61, pp. 227--261, 2002.

\bibitem{MVB}
M.~Van den Bergh.
\newblock Three dimensional flops and noncommutative rings.
\newblock {\em Duke Math.~J.~}, Vol. 122, pp. 423--455, 2004.

\bibitem{HRS}
D.~Happel, I.~Reiten, and S.~O. Smal$\o$.
\newblock {\em Tilting in abelian categories and quasitilted algebras}, Vol.
  120 of {\em Mem.~Amer.~Math.~Soc}.
\newblock 1996.

\bibitem{HL}
J.~Hu and W.~P. Li.
\newblock The {D}onaldson-{T}homas invariants under blowups and flops.
\newblock {\em preprint}.
\newblock arXiv:0505542.

\bibitem{HT2}
D.~Huybrechts and R.~P. Thomas.
\newblock Deformation-obstruction theory for complexes via
  {A}tiyah-{K}odaira-{S}pencer classes.
\newblock {\em Math.~Ann.=}, pp. 545--569, 2010.

\bibitem{Joy2}
D.~Joyce.
\newblock Configurations in abelian categories {I}\hspace{-.1em}{I}.
  {R}ingel-{H}all algebras.
\newblock {\em Advances in Math}, Vol. 210, pp. 635--706, 2007.

\bibitem{JS}
D.~Joyce and Y.~Song.
\newblock A theory of generalized {D}onaldson-{T}homas invariants.
\newblock {\em preprint}.
\newblock arXiv:0810.5645.

\bibitem{Ka2}
Y.~Kawamata.
\newblock On the cone of divisors of {C}alabi-{Y}au fiber spaces.
\newblock {\em Internat.~J.~Math}, Vol.~5, pp. 665--687, 1997.

\bibitem{Kawaflo}
Y.~Kawamata.
\newblock Flops connect minimal models.
\newblock {\em Publ.~Res.~Inst.~Math.~Sci.~}, Vol.~44, pp. 419--423, 2008.

\bibitem{Komi}
Y.~Konishi and S.~Minabe.
\newblock Flop invariance of the topological vertex.
\newblock {\em Internat.~J.~Math}, Vol.~19, pp. 27--45, 2008.

\bibitem{K-S}
M.~Kontsevich and Y.~Soibelman.
\newblock Stability structures, motivic {D}onaldson-{T}homas invariants and
  cluster transformations.
\newblock {\em preprint}.
\newblock arXiv:0811.2435.

\bibitem{LP}
M.~Levine and R.~Pandharipande.
\newblock Algebraic cobordism revisited.
\newblock {\em Invent.~Math.~}, Vol. 176, pp. 63--130, 2009.

\bibitem{LiRu}
A.-M. Li and Y.~Ruan.
\newblock Symplectic surgery and {G}romov-{W}itten invariants of {C}alabi-{Y}au
  3-folds.
\newblock {\em Invent.~Math}, Vol. 145, pp. 151--218, 2001.

\bibitem{Li}
J.~Li.
\newblock Zero dimensional {D}onaldson-{T}homas invariants of threefolds.
\newblock {\em Geom.~Topol.~}, Vol.~10, pp. 2117--2171, 2006.

\bibitem{LIE}
M.~Lieblich.
\newblock Moduli of complexes on a proper morphism.
\newblock {\em J.~Algebraic Geom.~}, Vol.~15, pp. 175--206, 2006.

\bibitem{LY}
C.~H. Liu and S.~T. Yau.
\newblock Transformations of algebraic {G}romov-{W}itten invariants of
  three-folds under flops and small extremal transitions, with an appendix from
  the stringy and the symplectic viewpoint.
\newblock {\em preprint}.
\newblock arXiv:0505084.

\bibitem{MNOP}
D.~Maulik, N.~Nekrasov, A.~Okounkov, and R.~Pandharipande.
\newblock Gromov-{W}itten theory and {D}onaldson-{T}homas theory. {I}.
\newblock {\em Compositio.~Math}, Vol. 142, pp. 1263--1285, 2006.

\bibitem{Morr}
D.~R. Morrison.
\newblock Beyond the {K}$\ddot{\textrm{a}}$hler cone.
\newblock {\em Proceedings of the Hirzebruch 65 Conference on Algebraic
  Geometry (Ramat Gan,1993)}, pp. 361--376.

\bibitem{MR}
S.~Mozgovoy and M.~Reineke.
\newblock On the noncommutative {D}onaldson-{T}homas invariants arising from
  brane tilings.
\newblock {\em preprint}.
\newblock arXiv:0809.0117.

\bibitem{Ng}
K.~Nagao.
\newblock Derived categories of small toric {C}alabi-{Y}au 3-folds and counting
  invariants.
\newblock {\em preprint}.
\newblock arXiv:0809.2994.

\bibitem{NN}
K.~Nagao and H.~Nakajima.
\newblock Counting invariant of perverse coherent sheaves and its
  wall-crossing.
\newblock {\em IMRN (to appear)}.
\newblock arXiv:0809.2992.

\bibitem{PT}
R.~Pandharipande and R.~P. Thomas.
\newblock Curve counting via stable pairs in the derived category.
\newblock {\em Invent.~Math.~}, Vol. 178, pp. 407--447, 2009.

\bibitem{Sz}
B.~Szendr{\H o}i.
\newblock Non-commutative {D}onaldson-{T}homas theory and the conifold.
\newblock {\em Geom.~Topol.~}, Vol.~12, pp. 1171--1202, 2008.

\bibitem{Tho}
R.~P. Thomas.
\newblock Stability conditions and the braid groups.
\newblock {\em Comm.~Anal.~Geom.~}, Vol.~14, pp. 135--161, 2006.

\bibitem{ToBPS}
Y.~Toda.
\newblock Birational {C}alabi-{Y}au 3-folds and {BPS} state counting.
\newblock {\em Communications in Number Theory and Physics}, Vol.~2, pp.
  63--112, 2008.

\bibitem{Tst}
Y.~Toda.
\newblock Stability conditions and crepant small resolutions.
\newblock {\em Trans.~Amer.~Math.~Soc.}, Vol. 360, pp. 6149--6178, 2008.

\bibitem{Tolim}
Y.~Toda.
\newblock Limit stable objects on {C}alabi-{Y}au 3-folds.
\newblock {\em Duke Math.~J.~}, Vol. 149, pp. 157--208, 2009.

\bibitem{Tcurve1}
Y.~Toda.
\newblock Curve counting theories via stable objects~{I}: {DT/PT}
  correspondence.
\newblock {\em J.~Amer.~Math.~Soc.~}, Vol.~23, pp. 1119--1157, 2010.

\bibitem{Tolim2}
Y.~Toda.
\newblock Generating functions of stable pair invariants via wall-crossings in
  derived categories.
\newblock {\em Adv.~Stud.~Pure Math.~}, Vol.~59, pp. 389--434, 2010.
\newblock New developments in algebraic geometry, integrable systems and mirror
  symmetry (RIMS, Kyoto, 2008).

\bibitem{Young2}
B.~Young with an appndix~by J.~Bryan.
\newblock Generating functions for colored 3{D} {Y}oung diagrams and the
  {D}onaldson-{T}homas invariants of orbifolds.
\newblock {\em preprint}.
\newblock arXiv:0802.3948.

\bibitem{Young1}
B.~Young.
\newblock Computing a pyramid partition generating function with dimer
  shuffling.
\newblock {\em J.~Combin.~Theory Ser.~}, pp. 334--350, 2009.

\end{thebibliography}
\end{document}